\documentclass[11pt]{article}
\usepackage{amssymb}

\usepackage[dvips]{graphicx}
\usepackage{amsmath}

\newtheorem{theorem}{Theorem}[subsection]

\newtheorem{corollary}[theorem]{Corollary}

\newtheorem{definition}[theorem]{Definition}
\newtheorem{example}[theorem]{Example}

\newtheorem{lemma}[theorem]{Lemma}

\newtheorem{proposition}[theorem]{Proposition}

\newenvironment{proof}[1][Proof]{\textbf{#1.} }{\ \rule{0.5em}{0.5em}}

\begin{document}

\title{Crystal isomorphisms for irreducible highest weight $\mathcal{U}_{v}(%
\widehat{\mathfrak{sl}}_{e})$-modules of higher level}
\author{Nicolas Jacon\footnote{Universit\'e de Franche-Comt\'e,  UFR Sciences et Techniques, 16 route de
Gray, 25 030 Besan\c{c}on, France. Email: njacon@univ-fcomte.fr} and C\'edric
Lecouvey\footnote{ Universit\'e du Littoral, Centre Universitaire de la Mi-Voix
Maison de la Recherche Blaise Pascal, 50 rue F.Buisson B.P. 699 62228 Calais
Cedex, France. Email: Cedric.Lecouvey@lmpa.univ-littoral.fr}  }
\date{}
\maketitle

\begin{abstract}
We study the crystal graphs of irreducible $\mathcal{U}_{v}(\widehat{\mathfrak{sl%
}}_{e})$-modules of higher level $\ell $. Generalizing results of the first
author, we obtain a simple description of the bijections between the classes
of multipartitions which naturally label these graphs: the Uglov
multipartitions. By works of Ariki, Grojnowski and
Lascoux-Leclerc-Thibon, it is then known that these bijections permit also to link the distinct parametrizations of the simple modules in modular representation theory of
Ariki-Koike algebras. Our main tool is to make explicit an embedding of the $%
\mathcal{U}_{v}(\widehat{\mathfrak{sl}}_{e})$-crystals of level $\ell $ into $%
\mathcal{U}_{v}(\mathfrak{sl}_{\infty })$-crystals associated to highest weight
modules.
\end{abstract}

\section{Introduction}

Let $\mathcal{U}_{v}(\widehat{\mathfrak{sl}}_{e})$ be the affine quantum
group of type $A_{e-1}^{(1)}$ and $\{\Lambda _{i}\ |\ i\in \mathbb{Z}/e%
\mathbb{Z}\}$ its set of fundamental weights. Consider $\ell \in \mathbb{N}$
and $\underline{\frak{s}}=(\frak{s}_{0},....,\frak{s}_{\ell -1})\in (\mathbb{%
Z}/e\mathbb{Z)}^{\ell }$.\ Denote by $V_{e}(\Lambda _{\underline{\frak{s}}})$
the irreducible $\mathcal{U}_{v}(\widehat{\mathfrak{sl}}_{e})$-module of
highest weight $\displaystyle\Lambda _{\underline{\frak{s}}%
}=\sum_{i=0}^{\ell -1}\Lambda _{\frak{s}_{i}}$.\ The general theory of
Kashiwara provides a crystal basis and a global basis for $V_{e}(\Lambda _{%
\underline{\frak{s}}})$.\ The crystal basis of $V_{e}(\Lambda _{\underline{%
\frak{s}}})$ comes equipped with the crystal graph $B_{e}(\Lambda _{%
\underline{\frak{s}}})$ which encodes much information on the module
structure.

\noindent There are different possible realizations of $V_{e}(\Lambda _{%
\underline{\frak{s}}})$ depending on the choice of a representative $%
\underline{s}=(s_{0},...,s_{\ell -1})\in \mathbb{Z}^{\ell }$ of the class $%
\underline{\frak{s}}\in (\mathbb{Z}/e\mathbb{Z)}^{\ell }$.\ Indeed, to $%
\underline{s}$ is associated a Fock space $\frak{F}_{e}^{\underline{s}}$
which provides an explicit construction of $V_{e}(\Lambda _{\underline{\frak{%
s}}})$. We will denote it $V_{e}^{\underline{s}}(\Lambda _{\underline{\frak{s%
}}})$ and write $B_{e}^{\underline{s}}(\Lambda _{\underline{\frak{s}}})$ for
the corresponding crystal graph. The crystals $B_{e}^{\underline{s}}(\Lambda
_{\underline{\frak{s}}})$ with $\underline{s}\in \underline{\frak{s}}$ are
thus all isomorphic to the abstract crystal $B_{e}(\Lambda _{\underline{%
\frak{s}}})$.

The purpose of the paper is to make explicit the isomorphisms between the
crystals $B_{e}^{\underline{s}}(\Lambda _{\underline{\frak{s}}})$ when $%
\underline{s}$ runs over $\underline{\frak{s}}$. By the works of Ariki \cite{ari2}, Grojnowski \cite{Gro} and
Lascoux-Leclerc-Thibon \cite{LLT}, an important application of these isomorphisms is to provide bijections between the distinct parametrizations of the simple modules in modular representation theory of
Ariki-Koike algebras.

\noindent To be more precise, let $\eta $ be a primitive $e^{\text{th}}$%
-root of unity. Write $\mathcal{H=H}(\eta ;\eta ^{\frak{s}_{0}},...,\eta ^{%
\frak{s}_{\ell -1}})$ for the Ariki-Koike algebra defined over an
algebraically closed field $F$ of characteristic $0$. This algebra is
generated by $T_{0},\cdots ,T_{n-1}$ subject to the relations $(T_{0}-\eta ^{%
\frak{s}_{0}})...(T_{0}-\eta ^{\frak{s}_{\ell -1}})=0$, $(T_{i}-\eta
)(T_{i}+1)=0$, for $1\leq i\leq n$ and the type $B$ braid relations 
\begin{gather*}
(T_{0}T_{1})^{2}=(T_{1}T_{0})^{2},\quad T_{i}T_{i+1}T_{i}=T_{i+1}T_{i}T_{i+1}%
\text{ }(1\leq i<n), \\
T_{i}T_{j}=T_{j}T_{i}\text{ }(j\geq i+2).
\end{gather*}
The algebra $\mathcal{H}$ is not semisimple in general, and, by a deep
Theorem of Ariki, its representation theory is intimately connected to the
global bases of the irreducible $\mathcal{U}_{v}(\widehat{\mathfrak{sl}}_{e})
$-modules.\ In particular, the simple modules of $\mathcal{H}$ are labelled
by the vertices of any crystal $B_{e}^{\underline{s}}(\Lambda _{\underline{%
\frak{s}}})$ such that $\underline{s}\in \underline{\frak{s}}$.\ In fact the
Fock space $\frak{F}_{e}^{\underline{s}}$ admits a crystal basis indexed by
multipartitions of length $\ell .\;$This implies that the vertices of the
crystal $B_{e}^{\underline{s}}(\Lambda _{\underline{\frak{s}}})$ can be
identified with certain multipartitions of length $\ell $ which are called
the Uglov multipartitions. Note that when $\ell =2$, remarkable results of
Geck \cite{geck0} show that these multipartitions also naturally appear in
the context of Kazhdan-Lusztig theory and cellular structure of Hecke
algebras. Other connections between Uglov multipartitions and modular
representation theory of Ariki-Koike algebras are also known to hold for $%
\ell >2$ \cite{jac0}.

\noindent Although nice properties have been given in particular cases (see 
\cite{AKT} and \cite{AJ}), the combinatorics of the Uglov multipartitions
and their associated crystal graphs are not really well understood. For
example, we do not even have a non recursive characterization of the Uglov
multipartitions for any $\underline{s}\in \underline{\frak{s}}$.

This paper extends the results obtained in \cite{jac} for $\ell =2$ by
giving a combinatorial description of the isomorphisms between the crystals $%
B_{e}^{\underline{s}}(\Lambda _{\underline{\frak{s}}}).\;$Nevertheless, the
ideas we use are quite different.\ Indeed, we show that the crystals $B_{e}^{%
\underline{s}}(\Lambda _{\underline{\frak{s}}})$ can be embedded in crystals
corresponding to irreducible highest weight $\mathcal{U}_{v}(\mathfrak{sl}%
_{\infty })$-modules. This result allows us to prove that most of the
crystal isomorphisms in type $A_{e-1}^{(1)}$ can be derived from crystal
isomorphisms in type $A_{\infty }$. Now the combinatorial description of the
isomorphisms of $A_{\infty }$-crystals is very close to that of the
isomorphisms of $A_{r}$-crystals in finite rank $r.\;$This permits us to use
some elegant results of Nakayashiki and Yamada \cite{ny} on combinatorial $R$%
-matrices in type $A_{r}.$ One of the advantages of this new method is to
avoid cumbersome case by case verifications unavoidable in \cite{jac0}.

\noindent We would like to mention also that there is another way to realize
the abstract crystals $B_{e}(\Lambda _{\underline{\frak{s}}})$ by using Fock
spaces of level $\ell $ which are tensor products of Fock spaces of level $%
1.\;$The crystal $B_{e}^{a}(\Lambda _{\underline{\frak{s}}})$ so obtained
notably appears in the works by Ariki (see \cite{ari2}).\ The vertices of $%
B_{e}^{a}(\Lambda _{\underline{\frak{s}}})$ are parametrized by
multipartitions called the ``Kleshchev multipartitions'' \cite{Kle}.\ The
crystals $B_{e}^{\underline{s}}(\Lambda _{\underline{\frak{s}}})$ and $%
B_{e}^{a}(\Lambda _{\underline{\frak{s}}})$ do not coincide in general.\ In
particular, our method does not permit one to embed $B_{e}^{a}(\Lambda _{%
\underline{\frak{s}}})$ in a crystal of type $A_{\infty }$ (but see the
remark after Theorem \ref{th_embedd}).

\bigskip

The present paper is organized as follows. The second section is devoted to
the combinatorial description of certain isomorphisms of $\mathcal{U}_{v}(%
\frak{sl}_{\infty })$-crystals.\ In section $3$, we recall basic results on $%
\mathcal{U}_{v}(\widehat{\frak{sl}}_{e})$-crystals.\ By using two natural
parametrizations of the Dynkin diagram in type $A_{e-1}^{(1)}$, we link in
particular the two usual presentations of the crystals $B_{e}^{\underline{s}%
}(\Lambda _{\underline{\frak{s}}})$ which appear in the literature. We then
show in Section $4$ that the crystals $B_{e}^{\underline{s}}(\Lambda _{%
\underline{\frak{s}}})$ can be embedded in crystals corresponding to
irreducible highest weight $\mathcal{U}_{v}({\mathfrak{sl}}_{\infty })$%
-modules. This embedding allows us to give in Section $5$ a description of
the isomorphisms between the crystals $B_{e}^{\underline{s}}(\Lambda _{%
\underline{\frak{s}}})$ and to obtain another characterization for the sets
of Uglov multipartitions. This characterization does not necessitate an
induction on the sum of the parts of the Uglov multipartitions contrary to
the original one \cite{uglov}.

\section{Crystals in type $A_{\infty }$}

In this section, we study crystal isomorphisms in type $A_{\infty }$.

\subsection{Background on $\mathcal{U}_{v}(\frak{sl}_{\infty})$}

Let $\frak{sl}_{\infty }$ be the Lie algebra associated to the doubly
infinite Dynkin diagram in type $A_{\infty }$ (see \cite{Kac} and \cite{KAC2}%
). 
\begin{equation}
A_{\infty }:\cdot \cdot \cdot -\overset{-m}{\circ }-\overset{1-m}{\circ }%
-\cdot \cdot \cdot -\overset{-1}{\circ }-\overset{0}{\circ }-\overset{1}{%
\circ }-\cdot \cdot \cdot -\overset{m-1}{\circ }-\overset{m}{\circ }-\cdot
\cdot \cdot .  \label{Ainf}
\end{equation}
We denote by $\mathcal{U}_{v}(\frak{sl}_{\infty })$ the corresponding
quantum group and we write $\omega _{i},i\in \mathbb{Z}$ for the fundamental
weights of the corresponding root system. We associate to the sequence $%
\underline{s}=(s_{0},...,s_{\ell -1})\in \mathbb{Z}^{\ell }$ the dominant
weight $\omega _{\underline{s}}=\sum_{i=0}^{\ell -1}\omega _{s_{i}}.$ Then
the irreducible highest weight $\mathcal{U}_{v}(\frak{sl}_{\infty })$%
-modules are parametrized by the sequences $\underline{s}$ of arbitrary
length $\ell $. We denote by $V_{\infty }(\omega _{\underline{s}})$ the
irreducible $\mathcal{U}_{v}(\frak{sl}_{\infty })$-module of highest weight $%
\omega _{\underline{s}}.$ The module $V_{\infty }(\omega _{\underline{s}})$
admits a crystal basis.\ We refer the reader to \cite{Kan} for a complete
review on crystal bases.\ We write $B(\omega _{\underline{s}})$ for the
crystal graph corresponding to $V_{\infty }(\omega _{\underline{s}}).$ When $%
\underline{s}=(s_{0})$ we write for short $V_{\infty }(\omega _{s_{0}})$ and 
$B_{\infty }(\omega _{s_{0}})$ instead of $V(\omega _{\underline{s}})$ and $%
B(\omega _{\underline{s}})$.

In addition to the irreducible highest weight modules $V_{\infty }(\omega _{%
\underline{s}})$, it will be convenient to consider also irreducible modules 
$V_{\infty }(k)$ indexed by nonnegative integers which are not of highest
weight. By a column of height $k,$ we mean a column shaped Young diagram 
\begin{equation}
C= 
\begin{tabular}{|c|}
\hline
$x_{1}$ \\ \hline
$\cdot $ \\ \hline
$\cdot $ \\ \hline
$x_{k}$ \\ \hline
\end{tabular}
\label{colk}
\end{equation}
of height $k$ filled by integers $x_{i}\in \mathbb{Z}$ such that $%
x_{1}>\cdot \cdot \cdot >x_{k}.\;$When $a\in C$ and $b\notin C$, we write $%
C-\{a\}+\{b\}$ for the column obtained by replacing in $C$ the letter $a$ by
the letter $b.$ The module $V_{\infty }(k)$ is defined as the vector space
with basis $\mathcal{B}_{k}=\{v_{C}\mid C$ is a column of height $k$\}.\ The
actions of the Chevalley generators $e_{i},f_{i},k_{i}$, $i\in \mathbb{Z}$
of $\mathcal{U}_{v}(\frak{sl}_{\infty })$ are given by 
\begin{align*}
f_{i}(v_{C})& =\left\{ 
\begin{array}{l}
0\text{ if }\delta _{i}(C)=0\text{ or }\delta _{i+1}(C)=1 \\ 
v_{C^{\prime }}\text{ with }C^{\prime }=C\setminus \{i\}\cup \{i+1\}\text{
otherwise}
\end{array}
\right. , \\
e_{i}(v_{C})& =\left\{ 
\begin{array}{l}
0\text{ if }\delta _{i}(C)=1\text{ or }\delta _{i+1}(C)=0 \\ 
v_{C^{\prime }}\text{ with }C^{\prime }=C\setminus \{i+1\}\cup \{i\}\text{
otherwise}
\end{array}
\right. , \\
k_{i}(v_{C})& =v^{\delta _{i}(C)-\delta _{i+1}(C)}v_{C}.
\end{align*}
where for any $i\in \mathbb{Z}$, $\delta _{i}(C)=1$ if $i\in C$ and $\delta
_{i}(C)=0$ otherwise.

\bigskip

\noindent \textbf{Remark: }Consider $a,b$ two integers such that $a<b.\;$%
Denote by $[a,b[$ the set of integers $i$ such that $a\leq i\leq b-1$.\
Write $\mathcal{U}_{v}(\frak{sl}_{a,b})$ for the subalgebra of $\mathcal{U}%
_{v}(\frak{sl}_{\infty })$ generated by the Chevalley generators $%
e_{i},f_{i},k_{i}$ with $i\in \lbrack a,b[.\;$Then $\mathcal{U}_{v}(\frak{sl}%
_{a,b})$ can be identified with the quantum group associated to the Dynkin
diagram obtained by deleting the nodes $i\notin \lbrack a,b[$ in (\ref{Ainf}%
).\ In particular $\mathcal{U}_{v}(\frak{sl}_{a,b})$ is isomorphic to the
quantum group $\mathcal{U}_{v}(\frak{sl}_{d})$ with $d=b-a+1.\;$The subspace 
$V_{a,b}(k)$ of $V_{\infty }(k)$ generated by the basis vectors $v_{C}$ such
that $C$ contains only letters $x_{i}$ with $a\leq x_{i}\leq b$ has the
structure of a $\mathcal{U}_{v}(\frak{sl}_{a,b})$-module.\ Moreover $%
V_{a,b}(k)$ is isomorphic to the $k^{\text{th}}$ fundamental module of $%
\mathcal{U}_{v}(\frak{sl}_{a,b})$.

\bigskip

We follow the convention of \cite{kash} and consider $\mathcal{U}_{v}(\frak{%
sl}_{\infty })$ as a Hopf algebra with coproduct given by 
\begin{gather*}
\Delta (k_{i})=k_{i}\otimes k_{i}, \\
\Delta (e_{i})=e_{i}\otimes k_{i}+1\otimes e_{i}\text{ and }\Delta
(f_{i})=f_{i}\otimes 1+k_{i}^{-1}\otimes f_{i}.
\end{gather*}
Given $M_{1}$ and $M_{2}$ two $\mathcal{U}_{v}(\frak{sl}_{\infty })$-modules
with crystal graphs $B_{1}$ and $B_{2}$, the crystal graph structure on $%
B_{1}\otimes B_{2}$ is then given by 
\begin{align}
\widetilde{f_{i}}(u\otimes v)& =\left\{ 
\begin{tabular}{c}
$\widetilde{f}_{i}(u)\otimes v$ if $\varphi _{i}(v)\leq \varepsilon _{i}(u)$
\\ 
$u\otimes \widetilde{f}_{i}(v)$ if $\varphi _{i}(v)>\varepsilon _{i}(u)$%
\end{tabular}
\right. ,  \label{TENS1} \\
\widetilde{e_{i}}(u\otimes v)& =\left\{ 
\begin{tabular}{c}
$u\otimes \widetilde{e_{i}}(v)$ if $\varphi _{i}(v)\geq \varepsilon _{i}(u)$
\\ 
$\widetilde{e_{i}}(u)\otimes v$ if$\varphi _{i}(v)<\varepsilon _{i}(u)$%
\end{tabular}
\right. .  \label{TENS2}
\end{align}
Note that this convention is the reverse of that used in many references on
crystals bases (see for example \cite{HK}) but it is the natural one for
working with multipartitions.

\subsection{The crystals graphs of irreducible $\mathcal{U}_{v}(\frak{sl}%
_{\infty})$-modules\label{subsecRA}}

For any $s\in \mathbb{Z}$, the $\mathcal{U}_{v}(\frak{sl}_{\infty })$-module 
$V_{\infty }(\omega _{s})$ can be obtained as an irreducible component of
the Fock space $F_{\infty }(\omega _{s})$ defined by considering
semi-infinite wedge products of $V_{\infty }(1).$ The crystal $B_{\infty
}(\omega _{s})$ is identified with the graph whose vertices are the infinite
columns 
\begin{equation*}
\mathcal{C}= 
\begin{tabular}{|c|}
\hline
$x_{1}$ \\ \hline
$\cdot $ \\ \hline
$\cdot $ \\ \hline
$x_{k}$ \\ \hline
$s-k+1$ \\ \hline
$s-k$ \\ \hline
$\cdot \cdot \cdot $ \\ \hline
\end{tabular}
\end{equation*}
that is, the infinite columns shaped Young diagrams filled by decreasing
integers $x_{i}$ from top to bottom and such that $x_{k}=s-k+1$ for $k$
sufficiently large. Given $\mathcal{C}_{1}$ and $\mathcal{C}_{2}$ in $%
B_{\infty }(\omega _{s}),$ we have an arrow $\mathcal{C}_{1}\overset{i}{%
\rightarrow }\mathcal{C}_{2}$ if and only if $i\in \mathcal{C}_{1},$ $%
i+1\notin \mathcal{C}_{2}$ and $\mathcal{C}_{2}=\mathcal{C}%
_{1}-\{i\}+\{i+1\}.$ The highest weight vertex of $B_{\infty }(\omega _{s})$
is the column $\mathcal{C}^{(s)}$ such that $x_{k}=s-k+1$ for all $k\geq 1.$

\noindent We associate to the $\mathcal{U}_{v}(\frak{sl}_{\infty })$-module $%
V_{\infty }(k)$ a graph $B_{\infty }(k)$.\ Its vertices are the columns of
height $k$ (see (\ref{colk})) and we draw an oriented arrow $C_{1}\overset{i%
}{\rightarrow }C_{2}$ if and only if $i\in C_{1},$ $i+1\notin C_{2}$ and $%
C_{2}=C_{1}-\{i\}+\{i+1\}.$ The graph $B_{\infty }(k)$ can be regarded as
the crystal graph of $V_{\infty }(k)$.\ In fact, one can define a notion of
crystal basis for the module $V_{\infty }(k)$ despite the fact it is not of
highest weight. This can be done essentially by considering the direct limit
of the directed system formed by the crystal bases of the $\mathcal{U}_{v}(%
\frak{sl}_{a,b})$-modules $V_{a,b}(k)$ \cite{lec}. The corresponding crystal
graph then coincide with $B_{\infty }(k).$ In the sequel we will only need
the crystal $B_{\infty }(k)$ and not the whole crystal basis of $B_{\infty
}(k)$.

\noindent Consider an infinite column $\mathcal{C}\in B_{\infty }(\omega
_{s})$ with letters $x_{i}$, $i\geq 1$. We write $\pi _{a}(\mathcal{C)}$ for
the finite column obtained by deleting the infinite sequence of letters $%
x_{i}$ such that $x_{i}<a$ in $\mathcal{C}$. Note that the height of $\pi
_{a}(\mathcal{C)}$ depends then of the integer $a$ chosen.

\noindent Let $\underline{s}=(s_{0},...,s_{\ell -1})\in \mathbb{Z}^{\ell }$.
The abstract crystal $B_{\infty }(\omega _{\underline{s}})$ is then
isomorphic to the connected component $B_{\infty }^{_{\underline{s}}}(\omega
_{\underline{s}})$ of $B_{\infty }(\underline{s})=B_{\infty }(\omega
_{s_{0}})\otimes \cdot \cdot \cdot \otimes B_{\infty }(\omega _{s_{\ell -1}})
$ with highest weight vertex $b^{(\underline{s})}=\mathcal{C}%
^{(s_{0})}\otimes \cdot \cdot \cdot \otimes \mathcal{C}^{(s_{\ell -1})}$.
Consider $b=\mathcal{C}_{0}\otimes \cdot \cdot \cdot \otimes \mathcal{C}%
_{\ell -1}\in B_{\infty }^{_{\underline{s}}}(\omega _{\underline{s}})$ with $%
\mathcal{C}_{k}\in B_{\infty }(\omega _{s_{k}}),$ $k=0,...,\ell -1$. For any
integer $a$ and any $k=0,...,\ell -1,$ let $h_{k}$ be the height of the
finite column $\pi _{a}(\mathcal{C}_{k})$.$\;$We set 
\begin{equation*}
\pi _{a}(b)=\pi _{a}(\mathcal{C}_{0})\otimes \cdot \cdot \cdot \otimes \pi
_{a}(\mathcal{C}_{\ell -1})\in B_{\infty }(h_{0})\otimes \cdot \cdot \cdot
\otimes B_{\infty }(h_{\ell -1}).
\end{equation*}
Consider $b_{1}$ and $b_{2}$ two vertices of $B_{\infty }^{_{\underline{s}%
}}(\omega _{\underline{s}})$.\ Let $\widetilde{K}$ be any path between $b_{1}
$ and $b_{2}$ in $B_{\infty }^{_{\underline{s}}}(\omega _{\underline{s}}),$
that is any sequence of crystal operators such that $b_{1}=\widetilde{K}%
(b_{2}).$

\begin{lemma}
\label{lem_pia}With the previous notation, for any integer $a\geq 1$
sufficiently large, we have $\pi _{a}(b_{1})\in B_{\infty }(h_{0})\otimes
\cdot \cdot \cdot \otimes B_{\infty }(h_{\ell -1})$ and $\pi _{a}(b_{2})\in
B_{\infty }(h_{0})\otimes \cdot \cdot \cdot \otimes B_{\infty }(h_{\ell -1})$%
. In this case $\pi _{a}(b_{1})=\widetilde{K}(\pi _{a}(b_{2}))$ in $%
B_{\infty }(h_{0})\otimes \cdot \cdot \cdot \otimes B_{\infty }(h_{\ell -1})$%
.
\end{lemma}

\begin{proof}
We can choose $a$ sufficiently large so that the infinite columns appearing
in every vertex $b$ of the path joining $b_{1}$ to $b_{2}$ contain all the
letters $x<a$.\ Then, for each $k=0,...,\ell -1$, the height of the column
obtained by deleting the letters $x<a$ in the $k$-th column of $b$ does not
depend on $b.$ Set $h_{k}=s_{k}+1-a.$ The deleted letters $x<a$ do not
interfere during the computation of the crystal operators defining the path
between $b_{1}$ and $b_{2}.$ This implies that the path joining $\pi
_{a}(b_{1})$ to $\pi _{a}(b_{2})$ obtained by applying the same sequence $%
\widetilde{K}$ of crystal operators also exists in $B_{\infty
}(h_{0})\otimes \cdot \cdot \cdot \otimes B_{\infty }(h_{\ell -1}).$ Thus $%
\pi _{a}(b_{1})=\widetilde{K}(\pi _{a}(b_{2})).$
\end{proof}

\bigskip

The vertices of $B_{\infty }(\omega _{s})$ are also naturally labelled by
partitions. Recall that a partition $\lambda =\left( \lambda
_{1},...,\lambda _{p}\right) $ of length $p$ is a weakly decreasing sequence 
$\lambda _{1}\geq \cdot \cdot \cdot \geq \lambda _{p}$ of nonnegative
integers called the parts of $\lambda $. In the sequel we identify the
partitions having the same nonzero parts and denote by $\mathcal{P}$ the set
of all partitions. It is convenient to represent $\lambda $ by its Young
diagram.\ In the sequel we use the French convention for the Young diagram
(see example below). To any pair $(a,b)$ of integers, we associate the node $%
\gamma =(a,b)$ which is the box obtained after $a$ north moves and $b$ east
moves starting from the south-west box of $\lambda .$ Observe that the node $%
\gamma $ does not necessarily belong to $\lambda .\;$We say that $\gamma $
is removable when $\gamma =(a,b)\in \lambda $ and $\lambda \backslash
\{\gamma \}$ is a partition$.$ Similarly $\gamma $ is said addable when $%
\gamma =(a,b)\notin \lambda $ and $\lambda \cup \{\gamma \}$ is a partition.

\noindent We associate to the infinite column $\mathcal{C\in }B_{\infty
}(\omega _{s})$ with letters $x_{k}$ the partition $\lambda $ such that $%
\lambda _{k}=x_{k}-s+k-1.$ Since $\lambda _{k}=0$ for $k$ sufficiently
large, this permits us to label the vertices of $B_{\infty }(\omega _{s})$
by $\mathcal{P}$. Let $\gamma =(a,b)$ be a node for $\lambda \in B_{\infty
}(\omega _{s}).\;$The content of $\gamma $ is 
\begin{equation}
c(\gamma )=b-a+s.  \label{contain}
\end{equation}
In $B_{\infty }(\omega _{s})$, we have an arrow $\lambda \overset{i}{%
\rightarrow }\mu $ if and only if $\mu /\lambda =\gamma $ where $\gamma $ is
addable in $\lambda $ with $c(\gamma )=i$. For any partitions $\lambda
=(\lambda _{1},...,\lambda _{p})$ we set $\left| \lambda \right| =\lambda
_{1}+\cdot \cdot \cdot +\lambda _{p}.\;$Then $\left| \lambda \right| $ is
equal to the length of the directed path joining $\emptyset $ to $\lambda $
in the crystal $B_{\infty }(\omega _{s})$.

\begin{example}
\label{exam1} 
\begin{align*}
\text{To }\mathcal{C}& = 
\begin{tabular}{|c|}
\hline
$\mathtt{7}$ \\ \hline
$\mathtt{6}$ \\ \hline
$\mathtt{4}$ \\ \hline
$\mathtt{1}$ \\ \hline
$\mathtt{\bar{1}}$ \\ \hline
$\mathtt{\bar{2}}$ \\ \hline
$\cdot \cdot \cdot $ \\ \hline
\end{tabular}
\in B_{\infty }(\omega _{3})\text{ corresponds }\lambda = 
\begin{tabular}{|l|llll}
\cline{1-1}
$R$ & $A$ &  &  &  \\ \cline{1-3}
&  & \multicolumn{1}{|l}{$R$} & \multicolumn{1}{|l}{$\mathbf{A}$} &  \\ 
\cline{1-4}\cline{4-4}
&  & \multicolumn{1}{|l}{} & \multicolumn{1}{|l}{$R$} & \multicolumn{1}{|l}{}
\\ \cline{1-4}\cline{4-4}
&  & \multicolumn{1}{|l}{} & \multicolumn{1}{|l}{} & \multicolumn{1}{|l}{$A$}
\\ \cline{1-4}\cline{4-4}
\end{tabular}
. \\
\text{We have }\widetilde{f}_{4}(\mathcal{C)}& = 
\begin{tabular}{|c|}
\hline
$\mathtt{7}$ \\ \hline
$\mathtt{6}$ \\ \hline
$\mathtt{5}$ \\ \hline
$\mathtt{1}$ \\ \hline
$\mathtt{\bar{1}}$ \\ \hline
$\mathtt{\bar{2}}$ \\ \hline
$\cdot \cdot \cdot $ \\ \hline
\end{tabular}
\text{ or equivalently }\widetilde{f}_{4}(\lambda )= 
\begin{tabular}{|l|llll}
\cline{1-1}
$R$ & $A$ &  &  &  \\ \cline{1-4}
&  & \multicolumn{1}{|l}{\ \ \ } & \multicolumn{1}{|l}{$\mathbf{R}$} & 
\multicolumn{1}{|l}{} \\ \cline{1-4}
&  & \multicolumn{1}{|l}{} & \multicolumn{1}{|l}{} & \multicolumn{1}{|l}{}
\\ \cline{1-4}
&  & \multicolumn{1}{|l}{} & \multicolumn{1}{|l}{} & \multicolumn{1}{|l}{$A$}
\\ \cline{1-4}
\end{tabular}
\end{align*}
for the node $\gamma =(3,4)$ is addable in $\lambda $ with content $c(\gamma
)=4-3+3=4.$ Here we have written $\overline{k}$ instead of $-k$ for any
positive integer $k$.
\end{example}

\noindent Since $B_{\infty }^{_{\underline{s}}}(\omega _{\underline{s}})$ is
labelled by partitions, the crystal $B_{\infty }(\underline{s})$ is labelled
by multipartitions $\underline{\lambda }=(\lambda ^{(0)},...,\lambda ^{(\ell
-1)}).$ More precisely, to each vertex $b=\mathcal{C}_{0}\otimes \cdot \cdot
\cdot \otimes \mathcal{C}_{\ell -1}\in B_{\infty }(\underline{s})$ such that 
$\mathcal{C}_{k}\in B(\omega _{s_{k}})$ for $k=0,...,\ell -1,$ corresponds
the multipartition $\underline{\lambda }=(\lambda ^{(0)},...,\lambda ^{(\ell
-1)})$ where $\lambda ^{(k)}$ is obtained from $\mathcal{C}_{k}$ as
described above. By a part of the multipartition $\underline{\lambda }$, we
mean a part of one of the partitions $\lambda ^{(k)}$. The action of the
operators $\widetilde{e}_{i}$ and $\widetilde{f}_{i}$ on $b$ are deduced
from (\ref{TENS1}) and (\ref{TENS2}).\ One verifies easily that they can be
also obtained by applying the following algorithm.\ Let $w_{i}$ be the word
obtained by considering the letters $i$ and $i+1$ successively in the
columns $\mathcal{C}_{0},\mathcal{C}_{1},...,\mathcal{C}_{\ell -1}$ of $b$
when these columns are read from top to bottom and left to right. The word $%
w_{i}$ is called the $i$-signature of $b.\;$Encode in $w_{i}$ each integer $i
$ by a symbol $+$ and each integer $i+1$ by a symbol $-.$ Choose any factor
of consecutive $-+$ and delete it.\ Repeat this procedure until no factor $-+
$ can be deleted. The final sequence $\widetilde{w}_{i}$ is uniquely
determined and has the form $\widetilde{w}_{i}=+^{p}-^{q}.$ It is called the
reduced $i$ signature of $b.$ Then $\widetilde{f}_{i}(b)$ is obtained from $b
$ by replacing the integer $i$ corresponding to the rightmost symbol $+$ in $%
\widetilde{w}_{i}$ by $i+1.$ Similarly, $\widetilde{e}_{i}(b)$ is obtained
from $\lambda $ by replacing the integer $i+1$ corresponding to the leftmost
symbol $-$ in $\widetilde{w}_{i}$ by $i.$

\noindent One can describe similarly the action of $\widetilde{e}_{i}$ and $%
\widetilde{f}_{i}$ on $b$ considered as the multipartition $\underline{%
\lambda }=(\lambda ^{(0)},...,\lambda ^{(\ell -1)}).\;$This time, one
considers the word $W_{i}$ obtained by reading the addable and removable
nodes of $\underline{\lambda }$ with content $i$ successively in the
partitions $\lambda ^{(0)},...,\lambda ^{(\ell -1)}.\;$This reading is well
defined since a partition cannot both content addable and removable nodes
with content $i.$ Encode in $W_{i}$ each removable node by $R$ and each
addable node by $A.\;$Let $\widetilde{W}_{i}$ be the word obtained by
deleting the factors $RA$ successively in the encoding. One can write 
\begin{equation}
\widetilde{W}_{i}=A^{p}R^{q}.  \label{actinfinity}
\end{equation}
Then $\widetilde{f}_{i}(b)$ is obtained from $\underline{\lambda }$ by
adding the node $\gamma $ corresponding to the rightmost symbol $A$ in $%
\widetilde{W}_{i}.$

\subsection{The crystal isomorphism between $B_{\infty }(\protect\omega
_{k})\otimes B_{\infty }(\protect\omega _{l})$ and $B_{\infty }(\protect%
\omega _{l})\otimes B_{\infty }(\protect\omega _{k})$}

Consider $k$ and $l$ two integers.\ We denote by $\psi_{k,l}$ the crystal
graph isomorphism 
\begin{equation}
\psi_{k,l}:B_{\infty}(\omega_{k})\otimes B_{\infty}(\omega_{l})\overset{%
\simeq}{\rightarrow}B_{\infty}(\omega_{l})\otimes B_{\infty}(\omega_{k}).
\label{iso_psi}
\end{equation}
To give the explicit combinatorial description of $\psi_{k,l}$, we start by
considering the crystal graph isomorphism $\theta_{k,l}$ between $B_{\infty
}(k)\otimes B_{\infty}(l)$ and $B_{\infty}(l)\otimes B_{\infty}(k).$
Consider $C_{1}\otimes C_{2}$ in $B_{\infty}(k)\otimes B_{\infty}(l).$ We
are going to associate to $C_{1}\otimes C_{2}$ a vertex $C_{2}^{\prime}%
\otimes C_{1}^{\prime}\in B_{\infty}(l)\otimes B_{\infty}(k).$

\noindent Suppose first $k\geq l.\;$Consider $x_{1}=\min \{t\in C_{2}\}.\;$%
We associate to $x_{1}$ the integer $y_{1}\in C_{1}$ such that 
\begin{equation}
y_{1}=\left\{ 
\begin{array}{l}
\max \{z\in C_{1}\mid z\leq x_{1}\}\text{ if }\min \{z\in C_{1}\}\leq x \\ 
\max \{z\in C_{1}\}\text{ otherwise}
\end{array}
\right. .  \label{algo1}
\end{equation}
We repeat the same procedure to the columns $C_{1}\setminus \{y_{1}\}$ and $%
C_{2}\setminus \{x_{1}\}.$ By induction this yields a sequence $%
\{y_{1},...,y_{l}\}\subset C_{1}.\;$Then we define $C_{2}^{\prime }$ as the
column obtained by reordering the integers of $\{y_{1},...,y_{l}\}$ and $%
C_{1}^{\prime }$ as the column obtained by reordering the integers of $%
C_{1}\setminus \{y_{1},...,y_{l}\}+C_{2}.$

\noindent Now, suppose $k<l.\;$Consider $x_{1}=\min \{t\in C_{1}\}.\;$We
associate to $x_{1}$ the integer $y_{1}\in C_{2}$ such that 
\begin{equation}
y_{1}=\left\{ 
\begin{array}{l}
\min \{z\in C_{2}\mid x_{1}\leq z\}\text{ if }\max \{z\in C_{2}\}\geq x \\ 
\min \{z\in C_{2}\}\text{ otherwise}
\end{array}
\right. .  \label{algo2}
\end{equation}
We repeat the same procedure to the columns $C_{1}\setminus \{x_{1}\}$ and $%
C_{2}\setminus \{y_{1}\}$ and obtain a sequence $\{y_{1},...,y_{k}\}\subset
C_{2}.\;$Then we define $C_{1}^{\prime }$ as the column obtained by
reordering the integers of $\{y_{1},...,y_{k}\}$ and $C_{2}^{\prime }$ as
the column obtained by reordering the integers of $C_{2}\setminus
\{y_{1},...,y_{l}\}+C_{1}.$

\noindent We denote by $\theta_{k,l}$ the map defined from $B_{\infty
}(k)\otimes B_{\infty}(l)$ to $B_{\infty}(l)\otimes B_{\infty}(k)$ by
setting $\theta_{k,l}(C_{1}\otimes C_{2})=C_{2}^{\prime}\otimes
C_{1}^{\prime}.$

\begin{example}
\label{exam2}Consider $C_{1}= 
\begin{tabular}{|l|}
\hline
$\mathtt{9}$ \\ \hline
$\mathtt{8}$ \\ \hline
$\mathtt{7}$ \\ \hline
$\mathtt{5}$ \\ \hline
$\mathtt{4}$ \\ \hline
$\mathtt{2}$ \\ \hline
\end{tabular}
$ and $C_{2}=$%
\begin{tabular}{|l|}
\hline
$\mathtt{7}$ \\ \hline
$\mathtt{6}$ \\ \hline
$\mathtt{5}$ \\ \hline
$\mathtt{3}$ \\ \hline
$\mathtt{1}$ \\ \hline
\end{tabular}
.\ We obtain $\{y_{1},y_{2},y_{3},y_{4},y_{5}\}=\{9,2,4,5,7\}.\;$This gives $%
C_{2}^{\prime }=$%
\begin{tabular}{|l|}
\hline
$\mathtt{9}$ \\ \hline
$\mathtt{7}$ \\ \hline
$\mathtt{5}$ \\ \hline
$\mathtt{4}$ \\ \hline
$\mathtt{2}$ \\ \hline
\end{tabular}
and $C_{1}^{\prime }=$%
\begin{tabular}{|l|}
\hline
$\mathtt{8}$ \\ \hline
$\mathtt{7}$ \\ \hline
$\mathtt{6}$ \\ \hline
$\mathtt{5}$ \\ \hline
$\mathtt{3}$ \\ \hline
$\mathtt{1}$ \\ \hline
\end{tabular}
. Thus $\theta _{6,5}(C_{1}\otimes C_{2})=C_{2}^{\prime }\otimes
C_{1}^{\prime }.$ One can also easily verifies that $\theta
_{5,6}(C_{2}^{\prime }\otimes C_{1}^{\prime })=C_{1}\otimes C_{2}.$
\end{example}

\begin{proposition}
\label{iso_NY}The map $\theta _{k,l}$ is an isomorphism of $\mathcal{U}_{v}(%
\frak{sl}_{\infty })$-crystals that is, for any integer $i$ and any vertex $%
C_{1}\otimes C_{2}\in B_{\infty }(k)\otimes B_{\infty }(l)$, we have 
\begin{equation}
\theta _{k,l}\circ \widetilde{e}_{i}(C_{1}\otimes C_{2})=\widetilde{e}%
_{i}(C_{2}^{\prime }\otimes C_{1}^{\prime })\text{ and }\theta _{k,l}\circ 
\widetilde{f}_{i}(C_{1}\otimes C_{2})=\widetilde{f}_{i}(C_{2}^{\prime
}\otimes C_{1}^{\prime }).  \label{comm}
\end{equation}
\end{proposition}

\begin{proof}
Choose $a<b$ two integers such that the letters of $C_{1}$ and $C_{2}$
belong to $[a,b].\;$It follows from Proposition 3.21 of \cite{ny}, that the
isomorphism between the finite $\mathcal{U}_{v}(\frak{sl}_{a-1,b+2})$%
-crystals $B_{a-1,b+2}(k)\otimes B_{a-1,b+2}(l)$ and $B_{a-1,b+2}(l)\otimes
B_{a-1,b+2}(k)$ is given by the map $\theta _{k,l}.$ Since the actions of
the crystal operators $\widetilde{e}_{i}$ and $\widetilde{f}_{i}$ with $i\in
\lbrack a-1,b+2[$ on the crystals $B_{a-1,b+2}(l)$ and $B_{a-1,b+2}(k))$ can
be obtained from the crystal structures of $B_{\infty }(k)$ and $B_{\infty
}(l)$, this implies the commutation relations (\ref{comm}) for any $i\in
\lbrack a-1,b+2[.$ Now if $i\notin \lbrack a-1,b+2[$ we have $\widetilde{e}%
_{i}(C_{1}\otimes C_{2})=\widetilde{f}_{i}(C_{1}\otimes C_{2})=0$ because
the letters of $C_{1}$ and $C_{2}$ belong to $[a,b].$ Similarly $\widetilde{e%
}_{i}(C_{2}^{\prime }\otimes C_{1}^{\prime })=C_{2}^{\prime }\otimes
C_{1}^{\prime }=0$ because the letters of $C_{1}\cup C_{2}$ are the same as
those of $C_{1}^{\prime }\cup C_{2}^{\prime }.\;$Hence (\ref{comm}) holds
for any integer $i.$
\end{proof}

\bigskip

Now consider $\mathcal{C}_{1}\otimes \mathcal{C}_{2}\in B_{\infty }(\omega
_{k})\otimes B_{\infty }(\omega _{l}).$ Let $a$ be any integer such that $%
\mathcal{C}_{1}$ and $\mathcal{C}_{2}$ both contain all the integers $x<a.$
Set $C_{1}=\pi _{a}(\mathcal{C}_{1}),$ $C_{2}=\pi _{a}(\mathcal{C}_{2})$ and 
$\theta _{k,l}(C_{1}\otimes C_{2})=C_{2}^{\prime }\otimes C_{1}^{\prime }.$
Since $C_{1}$ and $C_{2}$ both contain only letters $x\geq a$, it is also
the case for the columns $C_{1}^{\prime }$ and $C_{2}^{\prime }$ by the
previous combinatorial description of the isomorphism $\theta _{k,l}$. Write 
$\mathcal{C}_{1}^{\prime }$ and $\mathcal{C}_{2}^{\prime }\ $for the
infinite columns obtained respectively from $C_{1}^{\prime }$ and $%
C_{2}^{\prime }$ by adding boxes containing all the letters $x<a.$ Then $%
\mathcal{C}_{1}^{\prime }\in B_{\infty }(\omega _{k})$ and $\mathcal{C}%
_{2}^{\prime }\in B_{\infty }(\omega _{l}).$ Indeed $\mathcal{C}_{1}\in
B_{\infty }(\omega _{k}),$ $\mathcal{C}_{2}\in B_{\infty }(\omega _{l})$ and 
$C_{1},C_{1}^{\prime }$ (resp.\ $C_{2},C_{2}^{\prime }$) have the same
height.

\begin{corollary}
\label{cor_iso}(of Proposition \ref{iso_NY}) For any $\mathcal{C}_{1}\otimes 
\mathcal{C}_{2}\in B_{\infty }(\omega _{k})\otimes B_{\infty }(\omega _{l})$
we have with the above notation $\psi _{k,l}(\mathcal{C}_{1}\otimes \mathcal{%
C}_{2})=\mathcal{C}_{2}^{\prime }\otimes \mathcal{C}_{1}^{\prime }$.
\end{corollary}

\begin{proof}
Recall that $\mathcal{C}^{(k)}\otimes \mathcal{C}^{(l)}$ is the highest
weight vertex of $B_{\infty }(\omega _{k})\otimes B_{\infty }(\omega
_{l}).\; $Write $\mathcal{C}_{1}\otimes \mathcal{C}_{2}=\widetilde{F}(%
\mathcal{C}^{(k)}\otimes \mathcal{C}^{(l)})$ where $\widetilde{F}$ is a
sequence of crystal operators $\widetilde{f}_{i},$ $i\in \mathbb{Z}$.\ The
crystal graph isomorphism $\psi _{k,l}$ must send the highest vertex of $%
B_{\infty }(\omega _{k})\otimes B_{\infty }(\omega _{l})$ on the highest
weight vertex of $B_{\infty }(\omega _{l})\otimes B_{\infty }(\omega _{k}).$
Thus, we have $\psi _{k,l}(\mathcal{C}^{(k)}\otimes \mathcal{C}^{(l)})=%
\mathcal{C}^{(l)}\otimes \mathcal{C}^{(k)}.$ In particular, the Corollary is
true for $\mathcal{C}_{1}\otimes \mathcal{C}_{2}=\mathcal{C}^{(k)}\otimes 
\mathcal{C}^{(l)}.$

\noindent Now consider $\mathcal{C}_{1}\otimes \mathcal{C}_{2}\in B_{\infty
}(\omega _{k})\otimes B_{\infty }(\omega _{l})$ and choose $a$ an integer
such that $\mathcal{C}_{1}$ and $\mathcal{C}_{2}$ both contain all the
integers $x<a.$ Set $C^{(k)}=\pi _{a}(\mathcal{C}^{(k)})$ and $C^{(l)}=\pi
_{a}(\mathcal{C}^{(l)}).$ Similarly write $C_{1}=\pi _{a}(\mathcal{C}_{1}),$ 
$C_{2}=\pi _{a}(\mathcal{C}_{2}).$ Since $\mathcal{C}_{1}\otimes \mathcal{C}%
_{2}=\widetilde{F}(\mathcal{C}^{(k)}\otimes \mathcal{C}^{(l)})$ we obtain
from Lemma \ref{lem_pia} the equality $C_{1}\otimes C_{2}=\widetilde{F}%
(C^{(k)}\otimes C^{(l)}).$ By definition of the crystal isomorphism $\theta
_{k,l},$ we thus have $C_{2}^{\prime }\otimes C_{1}^{\prime }=\widetilde{F}%
(C^{(l)}\otimes C^{(k)}).$ The columns $C^{(k)},C^{(l)},C_{1}^{\prime }$ and 
$C_{2}^{\prime }$ contains only letters $x\geq a.$ Hence we can write $%
\widetilde{F}=\widetilde{f}_{i_{1}}\cdot \cdot \cdot \widetilde{f}_{i_{r}}$
with $i_{m}\geq a$ for any $m\in \{1,...,r\}.$ Now the infinite columns $%
\mathcal{C}^{(l)},\mathcal{C}^{(k)},\mathcal{C}_{1}^{\prime }$ and $\mathcal{%
C}_{2}^{\prime }$ are respectively obtained from $C^{(k)},C^{(l)},C_{1}^{%
\prime }$ and $C_{2}^{\prime }$ by adding all the letters $x<a.$ We can thus
deduce from the equality $C_{2}^{\prime }\otimes C_{1}^{\prime }=\widetilde{F%
}(C^{(l)}\otimes C^{(k)}),$ the equality $\mathcal{C}_{2}^{\prime }\otimes 
\mathcal{C}_{1}^{\prime }=\widetilde{F}(\mathcal{C}^{(l)}\otimes \mathcal{C}%
^{(k)}).$ This implies that $\psi _{k,l}(\mathcal{C}_{1}\otimes \mathcal{C}%
_{2})=\mathcal{C}_{2}^{\prime }\otimes \mathcal{C}_{1}^{\prime }$.
\end{proof}

\begin{example}
Consider $\lambda ^{(1)}=(5,5,5,4,4,3)\in B_{\infty }(\omega _{4})$ and $%
\lambda ^{(2)}=(4,4,4,3,2)\in B_{\infty }(\omega _{3}).\;$The columns $%
\mathcal{C}_{1}$, $\mathcal{C}_{2}$ such that $\mathcal{C}_{1}\otimes 
\mathcal{C}_{2}=(\lambda ^{(1)},\lambda ^{(2)})\in B_{\infty }(\omega
_{4})\otimes B_{\infty }(\omega _{3})$ are 
\begin{equation*}
\mathcal{C}_{1}= 
\begin{tabular}{|c|}
\hline
$\mathtt{9}$ \\ \hline
$\mathtt{8}$ \\ \hline
$\mathtt{7}$ \\ \hline
$\mathtt{5}$ \\ \hline
$\mathtt{4}$ \\ \hline
$\mathtt{2}$ \\ \hline
$\mathtt{\bar{2}}$ \\ \hline
$\cdot \cdot \cdot $ \\ \hline
\end{tabular}
\text{ and }\mathcal{C}_{2}= 
\begin{tabular}{|c|}
\hline
$\mathtt{7}$ \\ \hline
$\mathtt{6}$ \\ \hline
$\mathtt{5}$ \\ \hline
$\mathtt{3}$ \\ \hline
$\mathtt{1}$ \\ \hline
$\mathtt{\bar{2}}$ \\ \hline
$\cdot \cdot \cdot $ \\ \hline
\end{tabular}
.\;\text{Hence }C_{1}= 
\begin{tabular}{|l|}
\hline
$\mathtt{9}$ \\ \hline
$\mathtt{8}$ \\ \hline
$\mathtt{7}$ \\ \hline
$\mathtt{5}$ \\ \hline
$\mathtt{4}$ \\ \hline
$\mathtt{2}$ \\ \hline
\end{tabular}
\text{ and }C_{2}= 
\begin{tabular}{|l|}
\hline
$\mathtt{7}$ \\ \hline
$\mathtt{6}$ \\ \hline
$\mathtt{5}$ \\ \hline
$\mathtt{3}$ \\ \hline
$\mathtt{1}$ \\ \hline
\end{tabular}
\text{ for }a=\bar{2}.
\end{equation*}
We deduce from Example \ref{exam2} that 
\begin{equation*}
\mathcal{C}_{2}^{\prime }= 
\begin{tabular}{|c|}
\hline
$\mathtt{9}$ \\ \hline
$\mathtt{7}$ \\ \hline
$\mathtt{5}$ \\ \hline
$\mathtt{4}$ \\ \hline
$\mathtt{2}$ \\ \hline
$\mathtt{\bar{2}}$ \\ \hline
$\cdot \cdot \cdot $ \\ \hline
\end{tabular}
\text{ and }\mathcal{C}_{1}^{\prime }= 
\begin{tabular}{|c|}
\hline
$\mathtt{8}$ \\ \hline
$\mathtt{7}$ \\ \hline
$\mathtt{6}$ \\ \hline
$\mathtt{5}$ \\ \hline
$\mathtt{3}$ \\ \hline
$\mathtt{1}$ \\ \hline
$\mathtt{\bar{2}}$ \\ \hline
$\cdot \cdot \cdot $ \\ \hline
\end{tabular}
.
\end{equation*}
Thus we can write $\psi _{4,3}(\lambda ^{(1)},\lambda ^{(2)})=(\mu
^{(2)},\mu ^{(1)})$ with $\mu ^{(1)}=(4,4,4,4,3,2)$ and $\mu
^{(2)}=(6,5,4,4,3).$
\end{example}

\noindent \textbf{Remark:} The columns $C_{1}$ and $C_{2}$ obtained in the
previous algorithm depend on the integer $a$ considered. Nevertheless, this
is not the case for the resulting infinite columns $\mathcal{C}_{1}^{\prime
} $ and $\mathcal{C}_{2}^{\prime }$ as long as $a<<0.$

\section{ Crystals in type $A_{e-1}^{(1)}$}

We now turn to the problem of studying the crystals in type $A_{e-1}^{(1)}$.
In this section, we recall and show basic facts on their combinatorial
descriptions.

\subsection{Background on $\mathcal{U}_{v}(\widehat{\frak{sl}}_{e})$}

In order to link the different labellings of the crystal graphs in type $%
A_{e-1}^{(1)}$ appearing in the literature, we shall need the two following
Dynkin diagrams 
\begin{equation}
A_{e-1}^{(1),+}: 
\begin{array}{ccccc}
&  & \overset{0}{\circ } &  &  \\ 
& \diagup &  & \diagdown &  \\ 
\overset{1}{\circ } & - & \cdot \cdot \cdot & - & \overset{e-1}{\circ }
\end{array}
\text{ and }A_{e-1}^{(1),-}: 
\begin{array}{ccccc}
&  & \overset{0}{\circ } &  &  \\ 
& \diagup &  & \diagdown &  \\ 
\overset{e-1}{\circ } & - & \cdot \cdot \cdot & - & \overset{1}{\circ }
\end{array}
.  \label{label_dynkin}
\end{equation}
Let $\mathcal{U}_{v}^{+}(\widehat{\frak{sl}}_{e})$ and $\mathcal{U}_{v}^{-}(%
\widehat{\frak{sl}}_{e})$ be the affine quantum groups defined respectively
from the root systems in type $A_{e-1}^{(1),+}$ and $A_{e-1}^{(1),-}$ (see
for example \cite{mathas}, Chapter 6). Observe that this notation, which is
not related to the triangular decomposition of$\mathcal{U}_{v}^{-}(\widehat{%
\frak{sl}}_{e})$, only means we are considering two copies of $\mathcal{U}%
_{v}^{+}(\widehat{\frak{sl}}_{e})$ with Dynkin diagrams $A_{e-1}^{(1),+}$
and $A_{e-1}^{(1),-}$.\ Write $\{\Lambda _{0}^{+},\Lambda
_{1}^{+},...,\Lambda _{e-1}^{+}\}$ and $\{\Lambda _{0}^{-},\Lambda
_{1}^{-},...,\Lambda _{e-1}^{-}\}$ for the dominant weights of the root
systems $A_{e-1}^{(1),+}$ and $A_{e-1}^{(1),-}.$ We have then $\Lambda
_{k}^{+}=\Lambda _{e-k}^{-}.\;$Write $\{E_{i}^{+},F_{i}^{+},K_{i}^{+}\mid
i=0,...,e-1\}$ and $\{E_{i}^{-},F_{i}^{-},K_{i}^{-}\mid i=0,...,e-1\}$ for
the sets of Chevalley generators respectively in $\mathcal{U}_{v}^{+}(%
\widehat{\frak{sl}}_{e})$ and $\mathcal{U}_{v}^{-}(\widehat{\frak{sl}}_{e})$%
.\ Let $v^{d^{+}}\in \mathcal{U}_{v}^{+}(\widehat{\frak{sl}}_{e})$ and $%
v^{d^{-}}\in \mathcal{U}_{v}^{-}(\widehat{\frak{sl}}_{e})$ be the quantum
derivation operators. Then the map 
\begin{equation}
\iota :\left\{ 
\begin{array}{l}
\mathcal{U}_{v}^{+}(\widehat{\frak{sl}}_{e})\rightarrow \mathcal{U}_{v}^{-}(%
\widehat{\frak{sl}}_{e}) \\ 
v^{d^{+}}\longmapsto v^{d-} \\ 
E_{i}^{+}\longmapsto E_{e-i}^{-},F_{i}^{+}\longmapsto F_{e-i}^{-}\text{ and }%
K_{i}^{+}\longmapsto K_{e-i}^{-}
\end{array}
\right.  \label{iso}
\end{equation}
is an isomorphism of algebras.

\noindent We associate to $\underline{s}=(s_{0},...,s_{\ell -1})\in \mathbb{Z%
}^{\ell }$ the dominants weights $\Lambda _{\underline{\frak{s}}%
}^{+}=\sum_{k=0}^{\ell -1}\Lambda _{s_{k}(\text{mod }e)}^{+}$ and $\Lambda _{%
\underline{\frak{s}}}^{-}=\sum_{k=0}^{\ell -1}\Lambda _{s_{k}(\text{mod }%
e)}^{-}$. Let $V_{e}(\Lambda _{\underline{\frak{s}}}^{+})$ (resp.\ $%
V_{e}(\Lambda _{\underline{\frak{s}}}^{-})$) be the irreducible $\mathcal{U}%
_{v}^{+}(\widehat{\frak{sl}}_{e})$-module (resp.\ $\mathcal{U}_{v}^{-}(%
\widehat{\frak{sl}}_{e})$-module) of highest weight $\Lambda _{\underline{%
\frak{s}}}^{+}$ (resp.\ $\Lambda _{\underline{\frak{s}}}^{-}$).\ These
modules can be constructed by using the Fock space representation $\frak{F}%
_{e}^{\underline{s}}$ of level $\ell .\;$Let $\Pi _{\ell ,n}$ be the set of
multipartitions $\underline{\lambda }=(\lambda ^{(0)},...,\lambda ^{(\ell
-1)})$ of length $\ell $ with rank $n,$ i.e.\ such that $\left| \lambda
^{(0)}\right| +\cdot \cdot \cdot \left| \lambda ^{(\ell -1)}\right| =n.\;$%
The Fock space $\frak{F}_{e}^{\underline{s}}$ is defined as the $\mathbb{C(}%
v)$-vector space generated by the symbols $\mid \underline{\lambda },%
\underline{s}\rangle $ with $\underline{\lambda }\in \Pi _{\ell ,n}$%
\begin{equation*}
\frak{F}_{e}^{\underline{s}}=\bigoplus_{n\geq 0}\bigoplus_{\underline{%
\lambda }\in \Pi _{\ell ,n}}\mathbb{C}(v)\mid \underline{\lambda },%
\underline{s}\rangle .
\end{equation*}
The Fock space $\frak{F}_{e}^{\underline{s}}$ can be endowed with the
structure of a $\mathcal{U}_{v}^{+}(\widehat{\frak{sl}}_{e})$-module or
equivalently with the structure of a $\mathcal{U}_{v}^{-}(\widehat{\frak{sl}}%
_{e})$-module.\ These actions are defined by introducing total orders $\prec
_{\underline{s}}^{+}$ and $\prec _{\underline{s}}^{-}$ on the $i$-nodes of
the multipartitions.\ Consider $\underline{\lambda }=(\lambda
^{(0)},...,\lambda ^{(\ell -1)})$ a multipartition and suppose $\underline{s}
$ (called the multicharge) is fixed. Then the nodes of $\underline{\lambda }$
can be identified with the triplet $\gamma =(a,b,c)$ where $c\in
\{0,...,\ell -1\}$ and $a,b$ are respectively the row and column indices of
the node $\gamma $ in $\lambda ^{(c)}.$ The content of $\gamma $ is the
integer $c\left( \gamma \right) =b-a+s_{c}$ and the residue $\mathrm{res(}%
\gamma )$ of $\gamma $ is the element of $\mathbb{Z}/e\mathbb{Z}$ such that 
\begin{equation}
\mathrm{res}(\gamma )\equiv c(\gamma )(\text{mod }e).  \label{res}
\end{equation}
We will say that $\gamma $ is an $i$-node of $\underline{\lambda }$ when $%
\mathrm{res}(\gamma )\equiv i(\text{mod }e).$ Let $\gamma
_{1}=(a_{1},b_{1},c_{1})$ and $\gamma _{2}=(a_{2},b_{2},c_{2})$ be two $i$%
-nodes in $\underline{\lambda }$. We define the total order $\prec _{%
\underline{s}}^{+}$ and $\prec _{\underline{s}}^{-}$ on the addable and
removable $i$-nodes of $\underline{\lambda }$ by setting 
\begin{align}
\gamma _{1}& \prec _{\underline{s}}^{+}\gamma _{2}\Longleftrightarrow
\left\{ 
\begin{array}{l}
c(\gamma _{1})<c(\gamma _{2})\text{ or} \\ 
c(\gamma _{1})=c(\gamma _{2})\text{ and }c_{1}<c_{2}
\end{array}
\right. ,  \label{def_order+} \\
\gamma _{1}& \prec _{\underline{s}}^{-}\gamma _{2}\Longleftrightarrow
\left\{ 
\begin{array}{l}
c(\gamma _{1})<c(\gamma _{2})\text{ or} \\ 
c(\gamma _{1})=c(\gamma _{2})\text{ and }c_{1}>c_{2}
\end{array}
\right. .  \label{def_order-}
\end{align}
Using these orders, it is possible to define an action of $\mathcal{U}%
_{v}^{+}(\widehat{\frak{sl}}_{e})$-module and an action of $\mathcal{U}%
_{v}^{-}(\widehat{\frak{sl}}_{e})$-module on $\frak{F}_{e}^{\underline{s}}$%
.\ These modules will be denoted by $\frak{F}_{e}^{\underline{s},+}$ and $%
\frak{F}_{e}^{\underline{s},-}.\;$For these actions the empty multipartition 
$\underline{\emptyset }=(\emptyset ,...,\emptyset )$ is a highest weight
vector respectively of highest weight $\Lambda _{\underline{\frak{s}}}^{+}$
and $\Lambda _{\underline{\frak{s}}}^{-}.$ We denote by $V_{e}^{\underline{s}%
}(\Lambda _{\underline{\frak{s}}}^{+})$ and $V_{e}^{\underline{s}}(\Lambda _{%
\underline{\frak{s}}}^{-})$ the irreducible components with highest weight
vector $\underline{\emptyset }$ in $\frak{F}_{e}^{\underline{s},+}$ and $%
\frak{F}_{e}^{\underline{s},-}$. We refer the reader to \cite{jim} for a
detailed exposition. Note that the modules $V_{e}^{\underline{s}}(\Lambda _{%
\underline{\frak{s}}}^{+})$ (resp.\ $V_{e}^{\underline{s}}(\Lambda _{%
\underline{\frak{s}}}^{-}))$ such that $\underline{s}\in \underline{\frak{s}}
$ are all isomorphic to the abstract module $V_{e}(\Lambda _{\underline{%
\frak{s}}}^{+})$ (resp.\ $V_{e}(\Lambda _{\underline{\frak{s}}}^{-}))$.\
However, the corresponding actions of the Chevalley operators do not
coincide in general. This will yield different parametrizations of the
associated crystals.

\subsection{Crystal graph of the Fock space $\frak{F}_{e}^{\protect%
\underline{s},+}$ and $\frak{F}_{e}^{\protect\underline{s},-}$}

The modules $\frak{F}_{e}^{\underline{s},+}$ and $\frak{F}_{e}^{\underline{s}%
,-}$ are integrable modules, thus by the general theory of Kashiwara, they
admit crystal bases.\ Write $B_{e}^{\underline{s},+}$ and $B_{e}^{\underline{%
s},-}$ for the crystal graphs corresponding to the action of $\mathcal{U}%
_{v}^{+}(\widehat{\frak{sl}}_{e})$ and $\mathcal{U}_{v}^{-}(\widehat{\frak{sl%
}}_{e})$ on $\frak{F}_{e}^{\underline{s}}$.\ These crystals are labelled by
multipartitions and their crystal graph structure can be explicitly
described by using the total orders $\prec _{\underline{s}}^{+}$ and $\prec
_{\underline{s}}^{-}.\;$Consider two multipartitions $\underline{\lambda },%
\underline{\mu }$ and an integer $i\in \{0,...,e-1\}.$ The crystal graph $%
B_{e}^{\underline{s}}(\Lambda _{\underline{\frak{s}}}^{+})$ (resp.\ $B_{e}^{%
\underline{s}}(\Lambda _{\underline{\frak{s}}}^{-}))$ of $V_{e}^{\underline{s%
}}(\Lambda _{\underline{\frak{s}}}^{+})$ (resp.\ $V_{e}^{\underline{s}%
}(\Lambda _{\underline{\frak{s}}}^{-}))$ is the connected component of $%
B_{e}^{\underline{s},+}$ (resp.\ $B_{e}^{\underline{s},-})$ whose highest
weight vertex is the empty multipartition.

\subsubsection{Crystal structure on $B_{e}^{\protect\underline{s},+}$ \label%
{act+}}

Consider the set of addable and removable $i$-nodes of $\underline{\lambda }$
(see Section \ref{subsecRA} for the definition of removable and addable
nodes). Let $w_{i}$ be the word obtained by writing the addable and
removable $i$-nodes of $\underline{\lambda }$ in \textit{decreasing} order
with respect to $\prec _{\underline{s}}^{+},$ next by encoding each addable $%
i$-node by the letter $A$ and each removable $i$-node by the letter $R$.\
Write $\widetilde{w}_{i}=A^{p}R^{q}$ for the word obtained from $w_{i}$ by
deleting as many of the factors $RA$ as possible. If $p>0,$ let $\gamma $ be
the rightmost addable $i$-node in $\widetilde{w}_{i}$.\ The node $\gamma $
is called the good $i$-node.\ Then we have an arrow $\underline{\lambda }%
\overset{i}{\rightarrow }\underline{\mu }$ in $B_{e}^{\underline{s},+}$ if
and only if $\underline{\mu }$ is obtained from $\underline{\lambda }$ by
adding the good $i$-node $\gamma .$

\subsubsection{Crystal structure on $B_{e}^{\protect\underline{s},-}$ \label%
{act-}}

Consider the set of addable and removable $i$-nodes of $\underline{\lambda }$%
. Let $w_{i}$ be the word obtained by writing these $i$-nodes in \textit{%
increasing} order with respect to $\prec _{\underline{s}}^{-},$ next by
encoding each addable $i$-node by the letter $A$ and each removable $i$-node
by the letter $R$.\ Write $\widetilde{w}_{i}=A^{p}R^{q}$ for the word
obtained from $w_{i}$ by deleting as many of the factors $RA$ as possible.
If $p>0,$ let $\gamma $ be the rightmost addable $i$-node in $\widetilde{w}%
_{i}$.\ Then we have an arrow $\underline{\lambda }\overset{i}{\rightarrow }%
\underline{\mu }$ in $B_{e}^{\underline{s},-}$ if and only if $\underline{%
\mu }$ is obtained from $\underline{\lambda }$ by adding the good $i$-node $%
\gamma .$

\subsubsection{Link between the crystals $B_{e}^{\protect\underline{s},+}$
and $B_{e}^{\protect\underline{s},-}$}

For any multicharge $\underline{s}=(s_{0},...,s_{\ell -1}),$ we denote by $%
\underline{s}^{\ast }$ the multicharge $\underline{s}^{\ast
}=(e-s_{0},...,e-s_{\ell -1}).\;$According to the isomorphism (\ref{iso}),
there should exist some bijections $\nu :B_{e}^{\underline{s},+}\rightarrow
B_{e}^{\underline{s}^{\ast },-}$ such that, given any two multipartitions $%
\underline{\lambda },\underline{\mu }$, 
\begin{equation}
\underline{\lambda }\overset{i}{\rightarrow }\underline{\mu }\text{ in }%
B_{e}^{\underline{s},+}\Longleftrightarrow \nu (\underline{\lambda })%
\overset{e-i}{\rightarrow }\nu (\underline{\mu })\text{ in }B_{e}^{%
\underline{s}^{\ast },-}.  \label{skewiso}
\end{equation}
Note that $\nu $ cannot be unique since each of the crystals $B_{e}^{%
\underline{s},+}$ and $B_{e}^{\underline{s}^{\ast },-}$ contains isomorphic
connected components. To each multipartition $\underline{\lambda }=(\lambda
^{(0)},...,\lambda ^{(\ell -1)})$ in $B_{e}^{\underline{s},+}$, we associate
its conjugate multipartition $\underline{\lambda }^{\prime }=$ $(\lambda
^{\prime (0)},...,\lambda ^{\prime (\ell -1)})$ in $B_{e}^{\underline{s}^*,+}$
where for any $k=0,...,\ell -1,$ $\lambda ^{\prime (k)}$ is the conjugate
partition of $\lambda ^{(k)}.$

\begin{proposition}
The map $\xi :\underline{\lambda }\mapsto \underline{\lambda }^{\prime }$
from $\frak{F}_{e}^{\underline{s},+}$ to $\frak{F}_{e}^{\underline{s}^{\ast
},-}$ is a bijection and satisfies (\ref{skewiso}).
\end{proposition}

\begin{proof}
Consider $\underline{\lambda }$ and $\underline{\lambda }^{\prime }$ as
vertices respectively of the crystals $\frak{F}_{e}^{\underline{s},+}$ and $%
\frak{F}_{e}^{\underline{s}^{\ast },-}$.\ Let $\gamma =(a,b,c)$ be a node
appearing at the right end of the $a$-th row of $\lambda ^{(c)}$ in $%
\underline{\lambda }.$ One associates to $\gamma $ the node $\gamma ^{\prime
}$ appearing on the top of the $a$-th column of $\lambda ^{\prime (k)}$ in $%
\underline{\lambda }^{\prime }.\;$Then by definition of the node $\gamma
^{\prime },$ we have $\gamma ^{\prime }=(b,a,c)$.\ This gives for the
content of the nodes $\gamma $ and $\gamma ^{\prime }$%
\begin{equation*}
c(\gamma )=b-a+s_{c}\text{ and }c(\gamma ^{\prime })=a-b+e-s_{c}.
\end{equation*}
Now observe that $\gamma $ is a removable (resp. addable) node if and only
if $\gamma ^{\prime }$ is a removable (resp. addable) node.\ We have thus 
\begin{equation}
\gamma \text{ is an }A\text{ (resp. }R)\text{ }i\text{-node }%
\Longleftrightarrow \gamma ^{\prime }\text{ is an }A\text{ (resp. }R)\text{ }%
(e-i)\text{-node.}  \label{equi}
\end{equation}
Let $w_{i}$ be the word obtained by writing the addable or removable $i$%
-nodes of $\underline{\lambda }$ in \textit{decreasing} order with respect
to $\prec _{\underline{s}}^{+}$ as in Section \ref{act+}. Similarly, let $%
w_{e-i}^{\prime }$ be the word obtained by writing the addable or removable $%
(e-i)$-nodes of $\underline{\lambda }^{\prime }$ in \textit{increasing}
order with respect to $\prec _{\underline{s}}^{+}$ as in Section \ref{act-}.
Write $w_{i}=\gamma _{1}\cdot \cdot \cdot \gamma _{r}$ where for any $%
m=1,...,r,$ $\gamma _{m}$ is an $i$-node that is addable or removable. Then
by definition of the order $\prec _{\underline{s}}^{+}$ and $\prec _{%
\underline{s}}^{-}$ (see (\ref{def_order+}) and (\ref{def_order-})), we have 
\begin{equation*}
w_{e-i}^{\prime }=\gamma _{1}^{\prime }\cdot \cdot \cdot \gamma _{r}^{\prime
}.
\end{equation*}
Write $\widetilde{w}_{i}=A^{p}R^{q}$ for the word obtained from $w_{i}$ by
the cancellation process of the pairs $RA$ . We deduce from (\ref{equi})
that the word $\widetilde{w}_{e-i}^{\prime }=(A^{\prime })^{p}(R^{\prime
})^{q}$ coincides with the word obtained from $w_{e-i}^{\prime }$ by the
cancellation process of the letters $RA.$ In particular, $\gamma $ is the
good $i$-node for $\underline{\lambda }$ if and only if $\gamma ^{\prime }$
is the good $(e-i)$-node for $\underline{\lambda }^{\prime }$. Hence the
bijection $\xi $ satisfies (\ref{skewiso}).
\end{proof}

\bigskip

Since $V_{e}^{\underline{s}}(\Lambda _{\underline{\frak{s}}}^{+})$ and $%
V_{e}^{\underline{s}^{\ast }}(\Lambda _{\underline{\frak{s}}^{\ast }}^{-})$
are the irreducible components with highest weight vector $\underline{%
\emptyset }$ in $\frak{F}_{e}^{\underline{s},+}$ and $\frak{F}_{e}^{%
\underline{s},-},$ their crystals graphs $B_{e}^{\underline{s}}(\Lambda _{%
\underline{\frak{s}}}^{+})$ and $B_{e}^{\underline{s}^{\ast }}(\Lambda _{%
\underline{\frak{s}}^{\ast }}^{-})$ can be realized as the connected
components of highest weight vertex $\underline{\emptyset }$ in $B_{e}^{%
\underline{s},+}$ and $B_{e}^{\underline{s}^{\ast },-}.$ Since $\xi (%
\underline{\emptyset })=\underline{\emptyset },$ we derive from the previous
proposition the equivalence 
\begin{equation}
\underline{\lambda }\overset{i}{\rightarrow }\underline{\mu }\text{ in }%
B_{e}^{\underline{s}}(\Lambda _{\underline{\frak{s}}}^{+})%
\Longleftrightarrow \xi (\underline{\lambda })\overset{e-i}{\rightarrow }\xi
(\underline{\mu })\text{ in }B_{e}^{\underline{s}^{\ast }}(\Lambda _{%
\underline{\frak{s}}^{\ast }}^{-}).  \label{passage}
\end{equation}

\begin{definition}
The vertices of $B_{e}^{\underline{s}}(\Lambda _{\underline{\frak{s}}}^{+})$
and $B_{e}^{\underline{s}}(\Lambda _{\underline{\frak{s}}}^{-})$ are called
the Uglov multipartitions.
\end{definition}

\begin{example}
\label{exa_restr}Suppose $\ell =1$ and $\underline{s}=(s).$

\begin{itemize}
\item  The vertices of $B_{e}^{\underline{s}}(\Lambda _{\underline{\frak{s}}%
}^{+})$ are the $e$-restricted partitions, that is the partitions $\lambda
=(\lambda _{1},...,\lambda _{p})$ such that $\lambda _{i}-\lambda _{i+1}\leq
e-1$ for any $i=1,...,p-1.$

\item  The vertices of $B_{e}^{\underline{s}}(\Lambda _{\underline{\frak{s}}%
}^{-})$ are the $e$-regular partitions, that is the partitions $\lambda \ $%
with at most $e-1$ parts equal.
\end{itemize}
\end{example}

\noindent\textbf{Remarks:}

\noindent $\mathrm{(i):}$ The crystals $B_{e}^{\underline{s}}(\Lambda _{%
\underline{\frak{s}}}^{+})$ are essentially those used in \cite{ari2}
whereas the crystals $B_{e}^{\underline{s}}(\Lambda _{\underline{\frak{s}}%
}^{-})$ appear in \cite{geck} and \cite{jac}.

\noindent $\mathrm{(ii):}$ Consider $\underline{s}=(s_{0},...,s_{\ell -1})$
and $\underline{s}^{\prime }=(s_{0}^{\prime },...,s_{\ell -1}^{\prime })$
two multicharges such that $s_{k}\equiv s_{k}^{\prime }$ for any $%
k=0,...,\ell -1.$ Then $\Lambda _{\underline{\frak{s}}}^{+}=\Lambda _{%
\underline{\frak{s}}^{\prime }}^{+},$ thus the crystals $B_{e}^{\underline{s}%
}(\Lambda _{\underline{\frak{s}}}^{+})$ and $B_{e}^{\underline{s}^{\prime
}}(\Lambda _{\underline{\frak{s}}}^{+})$ are isomorphic. The combinatorial
description of this isomorphism is complicated in general (but see Section 
\ref{sec_class}).\ Nevertheless, in the case when there exists an integer $d$
such that $s_{k}=s_{k}^{\prime }+de$ for any $k=0,...,\ell -1$, one derives
easily from the description of the crystal structure on $B_{e}^{\underline{s}%
}(\Lambda _{\underline{\frak{s}}}^{+})$ that the relevant isomorphism
coincide with the identity map. The situation is the same for the crystals $%
B_{e}^{\underline{s}}(\Lambda _{\underline{\frak{s}}}^{-})$ and $B_{e}^{%
\underline{s}^{\prime }}(\Lambda _{\underline{\frak{s}}}^{-}).$

\bigskip

When the multicharge $\underline{s}=(s_{0},...,s_{\ell -1})$ satisfies $%
0\leq s_{0}\leq \cdot \cdot \cdot \leq s_{\ell -1}\leq e-1,$ there exists a
combinatorial description of the Uglov multipartitions labelling $B_{e}^{%
\underline{s}}(\Lambda _{\underline{\frak{s}}}^{-})$ due to Foda, Leclerc,
Okado, Thibon and Welsh.

\begin{proposition}
\cite{FLOTW} \cite{jim}\label{prop_FLOTW} When $0\leq s_{0}\leq \cdot \cdot
\cdot \leq s_{\ell -1}\leq e-1,$ the multipartition $\underline{\lambda }%
=(\lambda ^{(0)},...,\lambda ^{(\ell -1)})$ belongs to $B_{e}^{\underline{s}%
}(\Lambda _{\underline{\frak{s}}}^{-})$ if and only if:

\begin{enumerate}
\item  $\underline{\lambda }$ is cylindric, that is, for every $k=0,...,\ell
-2$ we have $\lambda _{i}^{(k)}\geq \lambda _{i+s_{k+1}-s_{k}}^{(k+1)}$ for
all $i>1$ (the partitions are taken with an infinite numbers of empty parts)
and $\lambda _{i}^{(\ell -1)}\geq \lambda _{i+e+s_{0}-s_{\ell -1}}^{(0)}$
for all $i>1$

\item  for all $r>0,$ among the residues appearing at the right ends of the
length $r$ rows of $\underline{\lambda }$, at least one element of $%
\{0,1,...,e-1\}$ does not appear.
\end{enumerate}
\end{proposition}

\noindent\textbf{Remarks:}

\noindent $\mathrm{(i):}$ By using (\ref{passage}), it is easy to deduce a
combinatorial characterization of the Uglov multipartitions appearing in $%
B_{e}^{\underline{s}}(\Lambda _{\underline{\frak{s}}}^{+})$ when $0\leq
s_{\ell -1}\leq \cdot \cdot \cdot \leq s_{0}\leq e-1$. These multipartitions
are called the FLOTW multipartitions.

\noindent $\mathrm{(ii):}$ In the sequel we will essentially consider the
crystals $B_{e}^{\underline{s}}(\Lambda _{\underline{\frak{s}}}^{+})$
because they can be embedded in a natural way in $\mathcal{U}_{v}(\frak{sl}%
_{\infty })$-crystals.$\;$Thanks to (\ref{passage}), it will be easy to
translate our statements to make them compatible with the crystals $B_{e}^{%
\underline{s}}(\Lambda _{\underline{\frak{s}}}^{-})$.

\noindent $\mathrm{(iii):}$ There exists also another realization $%
B_{e}^{a}(\Lambda _{\underline{\frak{s}}}^{+})$ of the abstract crystal $%
B_{e}(\Lambda _{\underline{\frak{s}}}^{+})$ which is in particular used by
Ariki \cite{ari2}.\ This realization is obtained by defining the Fock space
of level $\ell $ as a tensor product of Fock spaces of level $1.\;$It is
suited to the Specht modules theory for Ariki-Koike algebras introduced by
Dipper, James and Mathas \cite{DJM}. Note that in the level $2$ case, Geck 
\cite{geck0} has generalized this theory by defining Specht modules adapted
to the definition of the Fock spaces used in this paper. It is expected that
similar results hold in higher level.

\noindent $\mathrm{(iv):}$ The multipartitions which label the vertices of
the crystals $B_{e}^{a}(\Lambda _{\underline{\frak{s}}}^{+})$ are called the
Kleshchev multipartitions.\ For any nonnegative integer $n,$ the subgraph of
the crystal $B_{e}^{a}(\Lambda _{\underline{\frak{s}}}^{+})$ containing all
the Kleshchev multipartitions of rank $m\leq n$ coincide with the
corresponding subgraph of $B_{e}^{\underline{s}}(\Lambda _{\underline{\frak{s%
}}}^{+})$ when the multicharge $\underline{s}$ satisfies $s_{i}-s_{i+1}>n-1$
for any $i=0,...,\ell -2$. As a consequence Kleshchev multipartitions are
particular cases of Uglov multipartitions.

\section{Embedding of $B_{e}^{\protect\underline{s}}(\Lambda _{\protect%
\underline{\frak{s}}}^{+})$ in $B_{\infty }^{\protect\underline{s}^{\Diamond
}}(\protect\omega _{\protect\underline{s}})$}

The aim of this section is to show that we have an embedding of crystals
from $B_{e}^{\underline{s}}(\Lambda _{\underline{\frak{s}}}^{+})$ to $%
B_{\infty }^{\underline{s}^{\Diamond }}(\omega _{\underline{s}})$ with $%
\underline{s}^{\Diamond }=(s_{\ell -1},...,s_{0}).$

\subsection{Embedding of $B_{e}^{\protect\underline{s}}(\Lambda _{\protect%
\underline{\frak{s}}}^{+})_{\leq n}$ in $B_{f}^{\protect\underline{s}%
}(\Lambda _{\protect\underline{\frak{s}}}^{+})_{\leq n}$}

In the sequel we denote by $\widetilde{E}_{i}$ and $\widetilde{F}_{i},$ $%
i=0,...,e-1$ the crystal operators corresponding to the $\mathcal{U}_{v}^{+}(%
\widehat{\frak{sl}}_{e})$-crystals. Consider $\underline{\lambda }$ and $%
\underline{\mu }$ two multipartitions in $\frak{F}_{e}^{\underline{s},+}$
and $i\in \{0,...,e-1\}$ such that $\widetilde{F}_{i}(\underline{\lambda })=%
\underline{\mu }.$ Then $\underline{\mu }$ is obtained by adding an $i$-node 
$\gamma $ in $\underline{\lambda }.\;$Set $j=c(\gamma ).$ In the sequel we
slightly abuse the notation and write $\widetilde{F}_{j}(\underline{\lambda }%
)=\underline{\mu }$ although $j$ does not belong to $\{0,...,e-1\}$ in
general. Note that $j$ is uniquely determined from the equality $\widetilde{F%
}_{i}(\underline{\lambda })=\underline{\mu }.$ Observe that with this
convention, an arrow in $\frak{F}_{e}^{\underline{s},+}$ can be labelled by
any integer.\ To recover the original labelling of a $\mathcal{U}_{v}^{+}(%
\widehat{\frak{sl}}_{e})$-crystal, it suffices to read these labels modulo $e
$. With this convention, when there exists an arrow from the multipartition $%
\underline{\lambda }$ to the multipartition $\underline{\mu }$ in \textit{%
both }$F_{e}^{\underline{s},+}$\textit{\ and }$F_{f}^{\underline{s},+}$
(with $e\neq f)$, this arrows can be pictured $\underline{\lambda }\overset{j%
}{\rightarrow }\underline{\mu }$ where $j$ is the content of the node added
to $\underline{\lambda }$ to obtain $\underline{\mu }$. In particular the
label so obtained is independent of $e$ and $f$ and it makes sense to write $%
\widetilde{F}_{j}(\underline{\lambda })=\underline{\mu }$ in $B_{e}^{%
\underline{s}}(\Lambda _{\underline{\frak{s}}}^{+})$ and $B_{f}^{\underline{s%
}}(\Lambda _{\underline{\frak{s}}}^{+})$.

\noindent For any multicharge $\underline{s}=(s_{0},...,s_{\ell -1})$ in $%
\mathbb{Z}^{\ell },$ we set $\left\| \underline{s}\right\| =\max \{\left|
s_{k}\right| \mid k=0,...,\ell -1\}$ where $\left| s_{k}\right| =s_{k}$ if $%
s_{k}$ is nonnegative and $\left| s_{k}\right| =-s_{k}$ otherwise. The
following proposition is a generalization in level $\ell $ of Proposition
4.1 in \cite{jac}.

\begin{proposition}
\label{embed(e,f)}Let $n$ be a nonnegative integer. Consider $\underline{%
\lambda }$ a multipartition in $B_{e}^{\underline{s}}(\Lambda _{\underline{%
\frak{s}}}^{+})$ such that $\underline{\lambda }=\widetilde{F}_{j_{1}}\cdot
\cdot \cdot \widetilde{F}_{j_{n}}(\underline{\emptyset }).$ Then $\left| 
\underline{\lambda }\right| =n$. Moreover, for any integer $f\geq n+\left\| 
\underline{s}\right\| ,$ we have $\underline{\lambda }=\widetilde{F}%
_{j_{1}}\cdot \cdot \cdot \widetilde{F}_{j_{n}}(\underline{\emptyset })$ in $%
B_{f}^{\underline{s}}(\Lambda _{\underline{\frak{s}}}^{+})$.
\end{proposition}

\begin{proof}
Since each crystal operator $\widetilde{F}_{j_{k}}$ adds a node on the
multipartition $\widetilde{F}_{j_{k+1}}\cdot \cdot \cdot \widetilde{F}%
_{j_{n}}(\underline{\emptyset })$, we have $\left| \lambda \right| =n$.

To prove the second assertion of our proposition, we proceed by induction on 
$n.\;$When $n=0,$ the proposition is immediate. Suppose our statement true
for any multipartition $\underline{\lambda }$ with $\left| \underline{%
\lambda }\right| =n-1.\;$Consider $\underline{\mu }\in B_{e}^{\underline{s}%
}(\Lambda _{\underline{\frak{s}}}^{+})$ with $\left| \underline{\mu }\right|
=n$.

\noindent With the above convention, there exist an integer $j$ and a
multipartition $\underline{\lambda }\in B_{e}^{\underline{s}}(\Lambda _{%
\underline{\frak{s}}}^{+})$ such that $\widetilde{F}_{j}(\underline{\lambda }%
)=\underline{\mu }$ and $\underline{\mu }$ is obtained by adding a node with
content $j$ to $\underline{\lambda }.\;$Let $i\in \{0,...,e-1\}$ be such
that $i\equiv j(\text{mod }e).$ Write $w_{i}^{e}$ for the word obtained by
writing the addable or removable $i$-nodes of $\underline{\lambda }$ in 
\textit{decreasing} order with respect to $\prec _{\underline{s}}^{+}$ as in
Section \ref{act+}.

\noindent Choose any integer $f\geq n+ \left\| \underline{s}\right\| .$
Observe that each $j$-node $\gamma =(a,b,c)$ in $\underline{\lambda }$
verifies $c(\gamma )=\mathrm{res}(\gamma )$ when its residue is computed
modulo $f$. 
Indeed, we have $c(\gamma )=b-a+s_{c}$ and thus 
\begin{equation}
-f<1-n+s_{c}\leq c(\gamma )\leq n-1+s_{c}<f.  \label{c(gamma)}
\end{equation}
Write $w_{j}^{f}$ for the word obtained by writing the addable or removable $%
j$-nodes of $\underline{\lambda }$ in \textit{decreasing} order with respect
to $\prec _{\underline{s}}^{+}$. Then the nodes contributing to $w_{j}^{f}$
are the addable and removable nodes of $\underline{\lambda }$ with content $%
j.$ This implies that $w_{j}^{f}$ is a subword of $w_{i}^{e},$ that is each
node appearing in $w_{j}^{f}$ also appears in $w_{i}^{e}.$ If $\gamma $ is a
node in $w_{i}^{e}$ which does not belong to $w_{j}^{f},$ we must thus have $%
c(\gamma )\neq j.\;$Since the nodes of $w_{i}^{e}$ are ordered according to $%
\prec _{\underline{s}}^{+},$ we deduce that $w_{j}^{f}$ is a factor of $%
w_{i}^{e}$ (i.e., there is no node $\gamma $ such that $c(\gamma )\neq j$
between two nodes of $w_{j}^{f}$).

\noindent Write $\widehat{w}_{i}^{e}$ and $\widehat{w}_{j}^{f}$ respectively
for the words obtained from $w_{i}^{e}$ and $w_{j}^{f}$ by the cancellation
process of the factors $RA$. Since these words do not depend on the order of
the consecutive factors $RA$ which are deleted, $\widehat{w}_{j}^{e}$ is a
factor of $\widehat{w}_{i}^{f}.$ Let $\gamma _{0}$ be the good $i$-node in $%
\widehat{w}_{i}^{e}.$ Suppose that $\gamma _{0}$ is not a node contributing
to $\widehat{w}_{j}^{f}.$ Then, there exists a node $\delta _{0}$ in $%
w_{j}^{f}$ which can be paired with $\gamma _{0}$ to produce a factor $RA.\;$%
In this case, $\gamma _{0}$ and $\delta _{0}$ are nodes of $w_{i}^{e}$
because $w_{j}^{f}$ is a factor of $w_{i}^{e}.$ This gives a contradiction.\
Indeed, the word $w_{i}^{e}$ does not depend on the cancellation process for
the factors $RA.$ So, if we can pair $\gamma _{0}$ and $\delta _{0}$
together in order to obtain a factor $RA$ in $w_{j}^{f},$ we can do the same
pairing in $w_{i}^{e}$ (for $w_{j}^{f}$ is a factor of $w_{i}^{e}$) and $%
\gamma _{0}$ cannot be the good $i$-node in $\widehat{w}_{i}^{e}.$ Hence we
have shown that $\gamma _{0}$ is a $j$-node of $\widehat{w}_{j}^{f}.\;$Since 
$\widehat{w}_{j}^{f}$ is a factor of $\widehat{w}_{i}^{e},$ $\gamma _{0}$ is
the rightmost addable $j$-node in $\widehat{w}_{j}^{f},$ thus is the good $j$%
-node for $w_{j}^{f}$. By the induction hypothesis, one has $\underline{%
\lambda }=\widetilde{F}_{j_{1}}\cdot \cdot \cdot \widetilde{F}_{j_{n}}(%
\underline{\emptyset })$ in $B_{e}^{\underline{s}}(\Lambda _{\underline{%
\frak{s}}}^{+})$ and $B_{f}^{\underline{s}}(\Lambda _{\underline{\frak{s}}%
}^{+})$ because $f\geq n+\left\| \underline{s}\right\| >n-1+\left\| 
\underline{s}\right\| .$ Moreover, we have $\widetilde{F}_{j}(\underline{%
\lambda })=\underline{\mu }$ in $B_{e}^{\underline{s}}(\Lambda _{\underline{%
\frak{s}}}^{+})$ and $B_{f}^{\underline{s}}(\Lambda _{\underline{\frak{s}}%
}^{+})$ by the previous arguments. Thus $\underline{\mu }=\widetilde{F}_{j}%
\widetilde{F}_{j_{1}}\cdot \cdot \cdot \widetilde{F}_{j_{n}}(\underline{%
\emptyset })$ in $B_{e}^{\underline{s}}(\Lambda _{\underline{\frak{s}}}^{+})$
and $B_{f}^{\underline{s}}(\Lambda _{\underline{\frak{s}}}^{+}).$
\end{proof}

\bigskip

\noindent For any fixed nonnegative integer $n$, write $B_{e}^{\underline{s}%
}(\Lambda _{\underline{\frak{s}}}^{+})_{\leq n}=\{\underline{\lambda }\in
B_{e}^{\underline{s}}(\Lambda _{\underline{\frak{s}}}^{+})\mid \left| 
\underline{\lambda }\right| \leq n\}.$ We deduce from the above proposition,
that the identity map $\iota _{e,f}:\underline{\lambda }\longmapsto 
\underline{\lambda }$ from $B_{e}^{\underline{s}}(\Lambda _{\underline{\frak{%
s}}}^{+})_{\leq n}$ to $B_{f}^{\underline{s}}(\Lambda _{\underline{\frak{s}}%
}^{+})_{\leq n}$ is an embedding of crystals.

\subsection{Embedding of $B_{e}^{\protect\underline{s}}(\Lambda _{\protect%
\underline{\frak{s}}}^{+})$ in $B_{\infty }^{\protect\underline{s}^{\Diamond
}}(\protect\omega _{\protect\underline{s}})$}

For any multipartition $\underline{\lambda }=(\lambda ^{(0)},...,\lambda
^{(\ell -1)})$ we write $\underline{\lambda }^{\Diamond }=(\lambda ^{(\ell
-1)},...,\lambda ^{(0)}).$

\begin{proposition}
\label{embed(f,inf)}Consider a multicharge $\underline{s}$ and $f,n$ two
nonnegative integers such that $f\geq n+\left\| s\right\| .\;$Let $%
\underline{\lambda }$ be a multipartition in $B_{f}^{\underline{s}}(\Lambda
_{\underline{\frak{s}}}^{+})$ with $\left| \underline{\lambda }\right| =n.\;$%
Suppose $\underline{\lambda }=\widetilde{F}_{j_{1}}\cdot \cdot \cdot 
\widetilde{F}_{j_{n}}(\underline{\emptyset })$ in $B_{f}^{\underline{s}%
}(\Lambda _{\underline{\frak{s}}}^{+})$. Then we have $\underline{\lambda }%
^{\Diamond }=\widetilde{f}_{j_{1}}\cdot \cdot \cdot \widetilde{f}_{j_{n}}(%
\underline{\emptyset })$ in the $\mathcal{U}_{v}(\frak{sl}_{\infty })$%
-crystal $B_{\infty }^{\underline{s}^{\Diamond }}(\omega _{\underline{s}})$.
\end{proposition}

\begin{proof}
We proceed by induction on $n.\;$When $n=0,$ the proposition is immediate.
Suppose our statement true for any multipartition $\underline{\lambda }$
with $\left| \underline{\lambda }\right| =n.\;$Consider $\underline{\mu }\in
B_{f}^{\underline{s}}(\Lambda _{\underline{\frak{s}}}^{+})$ with $\left| 
\underline{\mu }\right| =n$.\ There exists an integer $j$ and a
multipartition $\underline{\lambda }\in B_{e}^{\underline{s}}(\Lambda _{%
\underline{\frak{s}}}^{+})$ such that $\widetilde{F}_{j}(\underline{\lambda }%
)=\underline{\mu }$ and $\underline{\mu }$ is obtained by adding a node with
content $j$ to $\underline{\lambda }.\;$Since $f\geq n+\left\| s\right\| ,$
we have $\mathrm{res}(\gamma )=c(\gamma )$ for each node $\gamma $ in $%
\underline{\lambda }$ as in (\ref{c(gamma)}).\ Hence the word $w_{j}^{f}$
obtained by writing the addable or removable $j$-nodes of $\underline{%
\lambda }$ in \textit{decreasing} order with respect to $\prec _{\underline{s%
}}^{+}$ contains exactly the addable and removable nodes of $\underline{%
\lambda }$ with content equal to $j.\;$Set $\underline{\lambda }=(\lambda
^{(0)},...,\lambda ^{(\ell -1)}).\;$By definition of the order $\prec _{%
\underline{s}}^{+}$ (see \ref{def_order+}), this means that $w_{j}^{f}$ is
the word obtained by reading the addable and removable nodes with content
equal to $j$ successively in the partitions $\lambda ^{(\ell -1)},\lambda
^{(\ell -2)},...,\lambda ^{(0)}$ of $\underline{\lambda }$.\ Moreover each
partition $\lambda ^{(c)},$ $c=0,...,\ell -1$ contains at most one addable
or removable node with content $j$ because the nodes with the same content
in $\lambda ^{(c)}$ must belong to the same diagonal.\ By the induction
hypothesis, we know that $\underline{\lambda }=\widetilde{F}_{j_{1}}\cdot
\cdot \cdot \widetilde{F}_{j_{n}}(\underline{\emptyset })$ in $B_{f}^{%
\underline{s}}(\Lambda _{\underline{\frak{s}}}^{+})$ and $\underline{\lambda 
}^{\Diamond }=\widetilde{f}_{j_{1}}\cdot \cdot \cdot \widetilde{f}_{j_{n}}(%
\underline{\emptyset })$ in $B_{\infty }^{\underline{s}^{\Diamond }}(\omega
_{\underline{s}}).\;$The previous arguments show that the word $w_{j}^{f}$
coincide with the word $W_{j}$ obtained by reading the addable and removable
nodes of content $j$ in $\underline{\lambda }^{\Diamond }$ as described in
Section \ref{subsecRA} just before (\ref{actinfinity}). Thus we obtain $%
\underline{\mu }=\widetilde{F}_{j}\widetilde{F}_{j_{1}}\cdot \cdot \cdot 
\widetilde{F}_{j_{n}}(\underline{\emptyset })$ in $B_{f}^{\underline{s}%
}(\Lambda _{\underline{\frak{s}}}^{+})$ and $\underline{\mu }^{\Diamond }=%
\widetilde{f}_{j}\widetilde{f}_{j_{1}}\cdot \cdot \cdot \widetilde{f}%
_{j_{n}}(\underline{\emptyset })$ in $B_{\infty }^{\underline{s}^{\Diamond
}}(\omega _{\underline{s}}).$
\end{proof}

\bigskip

\noindent Set $B_{\infty }^{\underline{s}^{\Diamond }}(\omega _{\underline{s}%
})_{\leq n}=\{\underline{\lambda }\in B_{\infty }^{\underline{s}^{\Diamond
}}(\omega _{\underline{s}})\mid \left| \underline{\lambda }\right| \leq n\}.$
We deduce from the previous proposition that the map 
\begin{equation*}
\varphi _{f,\infty }^{(n)}:\left\{ 
\begin{array}{l}
B_{e}^{\underline{s}}(\Lambda _{\underline{\frak{s}}}^{+})_{\leq
n}\rightarrow B_{\infty }^{\underline{s}^{\Diamond }}(\omega _{\underline{s}%
})_{\leq n} \\ 
\underline{\lambda }=(\lambda ^{(0)},...,\lambda ^{(\ell -1)})\longmapsto 
\underline{\lambda }^{\Diamond }=\lambda ^{(\ell -1)}\otimes \cdot \cdot
\cdot \otimes \lambda ^{(0)}
\end{array}
\right. 
\end{equation*}
is an embedding of crystals for any $f\geq n+\left\| s\right\| .$

\begin{theorem}
\label{th_embedd}Given any positive integer $e$ and any multicharge $%
\underline{s}$ the map 
\begin{equation*}
\varphi _{e,\infty }:\left\{ 
\begin{array}{l}
B_{e}^{\underline{s}}(\Lambda _{\underline{\frak{s}}}^{+})\rightarrow
B_{\infty }^{\underline{s}^{\Diamond }}(\omega _{\underline{s}}) \\ 
\underline{\lambda }=(\lambda ^{(0)},...,\lambda ^{(\ell -1)})\longmapsto 
\underline{\lambda }^{\Diamond }=\lambda ^{(\ell -1)}\otimes \cdot \cdot
\cdot \otimes \lambda ^{(0)}
\end{array}
\right. 
\end{equation*}
is an embedding of crystals : for any $\underline{\lambda }\in B_{e}^{%
\underline{s}}(\Lambda _{\underline{\frak{s}}}^{+})$ we have 
\begin{equation*}
\underline{\lambda }=\widetilde{F}_{j_{1}}\cdot \cdot \cdot \widetilde{F}%
_{j_{r}}(\underline{\emptyset })\text{ in }B_{e}^{\underline{s}}(\Lambda _{%
\underline{\frak{s}}}^{+})\Longrightarrow \underline{\lambda }^{\Diamond }=%
\widetilde{f}_{j_{1}}\cdot \cdot \cdot \widetilde{f}_{j_{r}}(\underline{%
\emptyset })\text{ in }B_{\infty }^{\underline{s}^{\Diamond }}(\omega _{%
\underline{s}}).
\end{equation*}
\end{theorem}

\begin{proof}
Consider $\underline{\lambda }=(\lambda ^{(0)},...,\lambda ^{(\ell -1)})\in
B_{e}^{\underline{s}}(\Lambda _{\underline{\frak{s}}}^{+})$ such that $%
\underline{\lambda }=\widetilde{F}_{j_{1}}\cdot \cdot \cdot \widetilde{F}%
_{j_{r}}(\underline{\emptyset })$ and set $\left| \underline{\lambda }%
\right| =n.\;$By Proposition \ref{embed(e,f)} we have $\underline{\lambda }=%
\widetilde{F}_{j_{1}}\cdot \cdot \cdot \widetilde{F}_{j_{r}}(\underline{%
\emptyset })$ in any crystal $B_{f}^{\underline{s}}(\Lambda _{\underline{%
\frak{s}}}^{+})$ with $f\geq n+\left\| s\right\| .$ Fix such an integer $f.\;
$We derive from Proposition \ref{embed(f,inf)} that 
\begin{equation*}
\underline{\lambda }^{\Diamond }=\widetilde{f}_{j_{1}}\cdot \cdot \cdot 
\widetilde{f}_{j_{r}}(\underline{\emptyset })\text{ in }B_{\infty }^{%
\underline{s}^{\Diamond }}(\omega _{\underline{s}}).
\end{equation*}
Clearly the map $\varphi _{e,\infty }$ is injective.\ Thus it is an
embedding from the $\mathcal{U}_{v}(\widehat{\frak{sl}}_{e})$-crystal $%
B_{e}^{\underline{s}}(\Lambda _{\underline{\frak{s}}}^{+})$ to the $\mathcal{%
U}_{v}(\frak{sl}_{\infty })$-crystal $B_{\infty }^{\underline{s}^{\Diamond
}}(\omega _{\underline{s}}).$
\end{proof}

\bigskip

\noindent \textbf{Remark: }According to the previous theorem, the crystal $%
B_{e}^{\underline{s}}(\Lambda _{\underline{\frak{s}}}^{+})$ can be embedded
in $B_{\infty }^{\underline{s}^{\Diamond }}(\omega _{\underline{s}}).$ The
situation is more complicated for the crystal $B_{e}^{a}(\Lambda _{%
\underline{\frak{s}}}^{+})$ labelled by Kleshchev multipartitions (see
Remark $\mathrm{(iv)}$ after Proposition \ref{prop_FLOTW}).\ Indeed, the
subcrystal $B_{e}^{a}(\Lambda _{\underline{\frak{s}}}^{+})_{\leq n}$ can be
embedded in a crystal $B_{\infty }(\omega _{\underline{s}(n)})$ where the
multicharge $\underline{s}(n)=(s_{0}(n),...,s_{l-1}(n))$ verifies $%
s_{k}(n)-s_{k+1}(n)>n-1$ for any $k=0,...,l-2.\;$ Since $s(n)$ depends on $n$%
, this procedure cannot provide an embedding of the whole crystal $%
B_{e}^{a}(\Lambda _{\underline{\frak{s}}}^{+})$ in a crystal $B_{\infty
}(\omega _{\underline{t}})$ where $\underline{t}$ is a fixed multicharge.

\section{Isomorphism class of a multipartition\label{sec_class}}

We can now use the above embedding to obtain a simple characterization of
the set of Uglov multipartitions.

\subsection{The extended affine symmetric group $\widehat{S}_{\ell }\label%
{subsec_aff}$}

We write $\widehat{S}_{\ell }$ for the extended affine symmetric group in
type $A_{\ell -1}.$ The group $\widehat{S}_{\ell }$ can be regarded as the
group generated by the elements $\sigma _{1},...,\sigma _{\ell -1}$ and $%
y_{0},....,y_{\ell -1}$ together with the relations 
\begin{equation*}
\sigma _{i}\sigma _{i+1}\sigma _{i}=\sigma _{i+1}\sigma _{i}\sigma
_{i+1},\quad \sigma _{i}\sigma _{j}=\sigma _{j}\sigma _{i}\text{ for }\left|
i-j\right| >1,\quad \sigma _{i}^{2}=1
\end{equation*}
with all indices in $\{1,...,\ell -1\}$ and 
\begin{equation*}
y_{i}y_{j}=y_{j}y_{i},\quad \sigma _{i}y_{j}=y_{j}\sigma _{i}\text{ for }%
j\neq i,i+1,\quad \sigma _{i}y_{i}\sigma _{i}=y_{i+1}.
\end{equation*}
We identify the subgroup of $\widehat{S}_{\ell }$ generated by the
transpositions $\sigma _{i},$ $i=1,...,\ell -1$ with the symmetric group $%
S_{\ell }$ of rank $\ell .\;$For any $i\in \{0,...,\ell -1\}$, we set $%
z_{i}=y_{1}\cdot \cdot \cdot y_{i}.$ Write also $\xi =\sigma _{\ell -1}\cdot
\cdot \cdot \sigma _{1}$ and $\tau =y_{\ell }\xi .$ Since $%
y_{i}=z_{i-1}^{-1}z_{i}$, $\widehat{S}_{\ell }$ is generated by the
transpositions $\sigma _{i}$ with $i\in \{1,...,\ell -1\}$ and the elements $%
z_{i}$ with $i\in \{1,...,\ell \}.$ Observe that for any $i\in \{1,...,\ell
-1\},$ we have 
\begin{equation}
z_{i}=\xi ^{\ell -i}\tau ^{i}.  \label{dec-zi}
\end{equation}
This implies that $\widehat{S}_{\ell }$ is generated by the transpositions $%
\sigma _{i}$ with $i\in \{1,...,\ell -1\}$ and $\tau .$

\noindent Consider $e$ a fixed positive integer. We obtain a faithful action
of $\widehat{S}_{\ell }$ on $\mathbb{Z}^{\ell }$ by setting for any $%
\underline{s}=(s_{0},...,s_{\ell -1})\in \mathbb{Z}^{\ell }$%
\begin{equation*}
\sigma _{i}(\underline{s})=(s_{0},...,s_{i},s_{i-1},...,s_{\ell -1})\text{
and }y_{i}(\underline{s})=(s_{0},...,s_{i-1},s_{i}+e,...,s_{\ell -1}).
\end{equation*}
Then $\tau (\underline{s})=(s_{1},s_{2},...,s_{\ell -1},s_{0}+e)$. We denote
by $\mathcal{C}(\underline{s})$ the orbit of the multicharge $\underline{s}$
under the action of $\widehat{S}_{\ell }$ on $\mathbb{Z}^{\ell }$. Clearly
each class $\mathcal{C}(\underline{s})$ contains a unique multicharge $%
\underline{\widetilde{s}}=(\widetilde{s}_{0},...,\widetilde{s}_{\ell -1})$
such that 
\begin{equation}
0\leq \widetilde{s}_{\ell -1}\leq \cdot \cdot \cdot \leq \widetilde{s}%
_{0}\leq e-1.  \label{fund}
\end{equation}
Hence the orbits $\mathcal{C}(\underline{s})$ are parametrized by the
multicharges verifying (\ref{fund}). Given any multicharge $\underline{s}%
=(s_{0},....,s_{\ell -1})\in \mathbb{Z}^{\ell }$, it is easy to determinate $%
w\in \widehat{S}_{\ell }$ such that $\underline{\widetilde{s}}=w(\underline{s%
}).$ To do this, we compute a sequence of multicharges as follows. Choose $%
k\in \mathbb{N}$ minimal to have $s_{i}+ke\geq 0$ for any $i=0,...,\ell -1$%
.\ Then $z_{\ell -1}^{k}(\underline{s})\in \mathbb{N}^{\ell }.$ Consider $%
\sigma \in S_{\ell }$ such that the coordinates of $\underline{s}^{(\ell
-1)}=\sigma z_{\ell -1}^{k}(\underline{s})$ weakly decrease. Write $r_{\ell
-2}$ for the quotient of the division of $\underline{s}_{\ell -2}^{(\ell
-2)} $ by $e$ and set $\underline{s}^{(\ell -3)}=z_{\ell -2}^{-r_{\ell -2}}(%
\underline{s}^{(\ell -2)}).\;$By induction one can compute a sequence $%
\underline{s}^{(\ell -2)},...,\underline{s}^{(0)}$ such that, for any $%
i=1,...,\ell -2,$ $\underline{s}^{(i-1)}=z_{i}^{-r_{i}}(\underline{s}^{(i)})$
where $r_{i}$ is the quotient of the division of ${s}_{i}^{(i)}-%
{s}_{i+1}^{(i)}$ by $e$. We have then $\underline{\widetilde{s}}=%
\underline{s}^{(0)}$ and 
\begin{equation}
\underline{\widetilde{s}}=w(\underline{s})=z_{0}^{-r_{0}}\cdot \cdot \cdot
z_{\ell -2}^{-r_{\ell -2}}\sigma z_{\ell -1}^{k}(\underline{s}).
\label{reduc}
\end{equation}

\subsection{Action of the transformations $s_{i}$ and $\protect\tau $ on a
multipartition\label{action_eleme}}

Consider a multicharge $\underline{s}=(s_{0},...,s_{\ell -1})$ and $w$ an
element of the extended affine symmetric group. Set $\underline{s}^{\prime
}=w(\underline{s}).$ Since the indices of the fundamental weights of $%
\mathcal{U}_{v}^{+}(\widehat{\frak{sl}}_{e})$ belong to $\mathbb{Z}/e\mathbb{%
Z},$ we have $\Lambda _{\underline{\frak{s}}^{\prime }}^{+}=\Lambda _{%
\underline{\frak{s}}}^{+}.$ This implies that the crystals graphs $B_{e}^{%
\underline{s}}(\Lambda _{\underline{\frak{s}}}^{+})$ and $B_{e}^{\underline{s%
}^{\prime }}(\Lambda _{\underline{\frak{s}}^{\prime }}^{+})$ are isomorphic.
Write $\Gamma _{\underline{s},\underline{s}^{\prime }}$ for the isomorphism
between $B_{e}^{\underline{s}}(\Lambda _{\underline{\frak{s}}}^{+})$ and $%
B_{e}^{\underline{s}^{\prime }}(\Lambda _{\underline{\frak{s}}^{\prime
}}^{+}).\;$Given a multipartition $\underline{\lambda }$ in $B_{e}^{%
\underline{s}}(\Lambda _{\underline{\frak{s}}}^{+}),$ we are going to see
how it is possible to determinate $\mu =\Gamma _{\underline{s},\underline{s}%
^{\prime }}(\underline{\lambda })$ in a non-inductive way, that is, without
computing a path joining $\underline{\lambda }$ to $\underline{\emptyset }$
in $B_{e}^{\underline{s}}(\Lambda _{\underline{\frak{s}}}^{+}).$ According
to Section \ref{subsec_aff}, $w$ decomposes as a product of the elements $%
\sigma _{i}$ $i=1,...,\ell -1$ and $\tau .$ We write for short $\Xi _{%
\underline{s}}$ and $\Sigma _{\underline{s},i}$ respectively for the crystal
graph isomorphisms $\Gamma _{\underline{s},\tau (\underline{s})}$ and $%
\Gamma _{\underline{s},\sigma _{i}(\underline{s})}.$ The following
proposition is a generalization of \cite[Prop. 3.1]{jac}.

\begin{proposition}
\label{prop_isotau}Consider $\underline{\lambda }=(\lambda
^{(0)},...,\lambda ^{(\ell -1)})$ a multipartition and $\underline{s}$ a
multicharge. Then 
\begin{equation*}
\Xi _{\underline{s}}(\underline{\lambda })=(\lambda ^{(1)},...,\lambda
^{(\ell -1)},\lambda ^{(0)}).
\end{equation*}
\end{proposition}

\begin{proof}
Set $\underline{s}=(s_{0},...,s_{\ell -1}).\;$Then $\tau (\underline{s}%
)=(s_{1},....,s_{\ell -1},s_{0}+e)$. Let $\underline{\lambda }=(\lambda
^{(0)},...,\lambda ^{(\ell -1)})$ a multipartition and set $\underline{%
\lambda }^{\#}=(\lambda ^{(1)},...,\lambda ^{(\ell -1)},\lambda ^{(0)})$.
Consider $i\in \{0,1,...,e-1\}$ and $\gamma _{1}=(a_{1},b_{1},c_{1}),$ $%
\gamma _{2}=(a_{2},b_{2},c_{2})$ two $i$-nodes of $\underline{\lambda }$.
Then $\gamma _{1}^{\#}=(a_{1},b_{1},c_{1}-1(\text{mod }e))$ and $\gamma
_{2}^{\#}=(a_{2},b_{2},c_{2}-1(\text{mod }e))$ are two $i$-nodes of $%
\underline{\lambda }^{\#}$. We then easily check that $\gamma _{2}\prec _{%
\underline{s}}^{+}\gamma _{1}$ if and only if $\gamma _{2}^{\#}\prec _{\tau (%
\underline{s})}^{+}\gamma _{1}^{\prime }\ .$ This implies that $\Xi _{%
\underline{s}}(\underline{\lambda })=\underline{\lambda }^{\#}.$
\end{proof}

\bigskip

Consider $\underline{\lambda }=(\lambda ^{(0)},...,\lambda ^{(\ell -1)})$ a
multipartition and $\underline{s}$ a multicharge such that $\underline{%
\lambda }\in B_{e}^{\underline{s}}(\Lambda _{\underline{\frak{s}}}^{+})$.
Then we know that $\underline{\lambda }^{\Diamond }=(\lambda ^{(\ell
-1)},...\lambda ^{(0)})\in B_{\infty }^{\underline{s}^{\Diamond }}(\omega _{%
\underline{s}}).$ Recall that, by definition of $B_{\infty }^{\underline{s}%
^{\Diamond }}(\omega _{\underline{s}})$, we can write $\underline{\lambda }%
^{\Diamond }$ $=\lambda ^{(\ell -1)}\otimes \cdot \cdot \cdot \otimes
\lambda ^{(0)}.\;$For any integer $i\in \{0,...,n-1\}.$ Set 
\begin{equation}
\psi _{s_{i+1},s_{i}}(\lambda ^{(i+1)}\otimes \lambda ^{(i)})=\widetilde{%
\lambda }^{(i)}\otimes \widetilde{\lambda }^{(i+1)}  \label{elem}
\end{equation}
where $\psi _{s_{i+1},s_{i}}$ is the crystal graph isomorphism defined in (%
\ref{iso_psi}).

\begin{proposition}
\label{prop_iso_sigmai}With the above notation, we have 
\begin{equation*}
\Sigma _{\underline{s},i}(\underline{\lambda })=(\lambda ^{(0)},...,%
\widetilde{\lambda }^{(i+1)},\widetilde{\lambda }^{(i)},...\lambda ^{(\ell
-1)})
\end{equation*}
that is $\Sigma _{\underline{s},i}(\underline{\lambda })$ is obtained by
replacing in $\underline{\lambda },$ $\lambda ^{(i)}$ by $\widetilde{\lambda 
}^{(i+1)}$ and $\lambda ^{(i+1)}$ by $\widetilde{\lambda }^{(i)}.$
\end{proposition}

\begin{proof}
We have to prove that the diagram 
\begin{equation}
\begin{array}{lll}
B_{e}^{\underline{s}}(\Lambda _{\underline{\frak{s}}}^{+}) & \overset{%
\varphi _{e,\infty }}{\rightarrow } & B_{\infty }(\omega _{s_{\ell
-1}})\otimes \cdot \cdot \cdot \otimes B_{\infty }(\omega _{s_{i+1}})\otimes
B_{\infty }(\omega _{s_{i}})\otimes \cdot \cdot \cdot \otimes B_{\infty
}(\omega _{s_{0}}) \\ 
\Sigma _{\underline{s},i}\downarrow  &  & \text{ \ \ \ \ \ \ \ \ \ \ \ \ \ \
\ \ \ \ \ \ \ \ \ \ \ \ \ \ }\downarrow \psi _{s_{i+1},s_{i}} \\ 
B_{e}^{\sigma _{i}(\underline{s})}(\Lambda _{\underline{\frak{s}}}^{+}) & 
\overset{\varphi _{e,\infty }}{\rightarrow } & B_{\infty }(\omega _{s_{\ell
-1}})\otimes \cdot \cdot \cdot \otimes B_{\infty }(\omega _{s_{i}})\otimes
B_{\infty }(\omega _{s_{i+1}})\otimes \cdot \cdot \cdot \otimes B_{\infty
}(\omega _{s_{0}})
\end{array}
\label{diag}
\end{equation}
commutes. Consider a multipartition $\underline{\lambda }\in B_{e}^{%
\underline{s}}(\Lambda _{\underline{\frak{s}}}^{+}).\;$Set $\underline{%
\lambda }=(\lambda ^{(0)},...,\lambda ^{(\ell -1)}).\;$Let $\widetilde{F}%
_{i_{1}},...,\widetilde{F}_{i_{r}}$ be a sequence of crystal operators such
that $\underline{\lambda }=\widetilde{F}_{i_{1}}\cdot \cdot \cdot \widetilde{%
F}_{i_{r}}(\underline{\emptyset }).$ Then we can consider the multipartition 
$\underline{\mu }=\widetilde{F}_{i_{1}},...,\widetilde{F}_{i_{r}}(\underline{%
\emptyset })$ in the crystal $B_{e}^{\sigma _{i}(\underline{s})}(\Lambda _{%
\underline{\frak{s}}}^{+})$. Observe first that 
\begin{equation}
\psi _{s_{i+1},s_{i}}\circ \varphi _{e,\infty }(\underline{\emptyset }%
)=\varphi _{e,\infty }\circ \Sigma _{\underline{s},i}(\underline{\emptyset }%
).  \label{eq_empty}
\end{equation}
Moreover, the maps $\varphi _{e,\infty },$ $\Sigma _{\underline{s},i}$ and $%
\psi _{s_{i+1},s_{i}}$ commute with the crystal operators. This permits us
to write 
\begin{equation*}
\psi _{s_{i+1},s_{i}}\circ \varphi _{e,\infty }(\underline{\lambda })=\psi
_{s_{i+1},s_{i}}\circ \varphi _{e,\infty }(\widetilde{F}_{i_{1}}\cdot \cdot
\cdot \widetilde{F}_{i_{r}}(\underline{\emptyset }))=\widetilde{f}%
_{i_{1}}\cdot \cdot \cdot \widetilde{f}_{i_{r}}(\psi _{s_{i+1},s_{i}}\circ
\varphi _{e,\infty }(\underline{\emptyset })).
\end{equation*}
One the other hand we have 
\begin{equation*}
\varphi _{e,\infty }\circ \Sigma _{\underline{s},i}(\underline{\lambda }%
)=\varphi _{e,\infty }\circ \Sigma _{\underline{s},i}(\widetilde{F}%
_{i_{1}}\cdot \cdot \cdot \widetilde{F}_{i_{r}}(\underline{\emptyset }))=%
\widetilde{f}_{i_{1}}\cdot \cdot \cdot \widetilde{f}_{i_{r}}(\varphi
_{e,\infty }\circ \Sigma _{\underline{s},i}(\underline{\emptyset })).
\end{equation*}
Hence we derive the equality $\psi _{s_{i+1},s_{i}}\circ \varphi _{e,\infty
}(\underline{\lambda })=\varphi _{e,\infty }\circ \Sigma _{\underline{s},i}(%
\underline{\lambda })$ from (\ref{eq_empty}). This shows that the diagram (%
\ref{diag}) commutes and establish our proposition.
\end{proof}

\begin{example}
Take $\ell =3.$ Suppose $\underline{s}=(4,0,1)$ and $\underline{\lambda }%
=(\lambda ^{(0)},\lambda ^{(1)},\lambda ^{(2)})$ with $\lambda
^{(0)}=(4,3,3,2),$ $\lambda ^{(1)}=(3,3,1)$ and $\lambda ^{(2)}=(5,3,2).$
Let us compute $\Sigma _{\underline{s},2}(\underline{\lambda })$.\ The
infinite columns associated to $\lambda ^{(1)}$ and $\lambda ^{(2)}$ are
respectively 
\begin{equation*}
\mathcal{C}_{1}=
\begin{tabular}{|c|}
\hline
$\mathtt{3}$ \\ \hline
$\mathtt{2}$ \\ \hline
$\mathtt{\bar{1}}$ \\ \hline
$\mathtt{\bar{3}}$ \\ \hline
$\mathtt{\bar{4}}$ \\ \hline
$\cdot \cdot \cdot $ \\ \hline
\end{tabular}
\text{ and }\mathcal{C}_{2}=
\begin{tabular}{|c|}
\hline
$\mathtt{6}$ \\ \hline
$\mathtt{3}$ \\ \hline
$\mathtt{1}$ \\ \hline
$\mathtt{\bar{2}}$ \\ \hline
$\mathtt{\bar{3}}$ \\ \hline
$\mathtt{\bar{4}}$ \\ \hline
$\cdot \cdot \cdot $ \\ \hline
\end{tabular}
.
\end{equation*}
For $a=\overline{3}$ the corresponding finite columns are 
\begin{equation*}
C_{1}=
\begin{tabular}{|c|}
\hline
$\mathtt{3}$ \\ \hline
$\mathtt{2}$ \\ \hline
$\mathtt{\bar{1}}$ \\ \hline
$\mathtt{\bar{3}}$ \\ \hline
\end{tabular}
\text{ and }C_{2}=
\begin{tabular}{|c|}
\hline
$\mathtt{6}$ \\ \hline
$\mathtt{3}$ \\ \hline
$\mathtt{1}$ \\ \hline
$\mathtt{\bar{2}}$ \\ \hline
$\mathtt{\bar{3}}$ \\ \hline
\end{tabular}
.
\end{equation*}
We have to determinate the image of $C_{2}\otimes C_{1}$ under the
isomorphism $\theta _{5,4}$ of Proposition \ref{iso_NY}. Note that the image
of $C_{1}\otimes C_{2}$ under $\theta _{4,5}$ is not relevant here because
we must take into account the swap $\diamond $.$\;$We obtain $%
\{y_{1},y_{2},y_{3},y_{4}\}=\{3,1,\bar{2},\bar{3}\}.\;$This gives $\theta
_{5,4}(C_{2}\otimes C_{1})=C_{1}^{\prime }\otimes C_{2}^{\prime }$ with $%
C_{1}^{\prime }=$%
\begin{tabular}{|c|}
\hline
$\mathtt{3}$ \\ \hline
$\mathtt{1}$ \\ \hline
$\mathtt{\bar{2}}$ \\ \hline
$\mathtt{\bar{3}}$ \\ \hline
\end{tabular}
and $C_{2}^{\prime }=
\begin{tabular}{|c|}
\hline
$\mathtt{6}$ \\ \hline
$\mathtt{3}$ \\ \hline
$\mathtt{2}$ \\ \hline
$\mathtt{\bar{1}}$ \\ \hline
$\mathtt{\bar{3}}$ \\ \hline
\end{tabular}
$. Hence $\psi _{1,0}(\mathcal{C}_{2}\otimes \mathcal{C}_{1})=\mathcal{C}%
_{1}^{\prime }\otimes \mathcal{C}_{2}^{\prime }$ where 
\begin{equation*}
\mathcal{C}_{1}^{\prime }=
\begin{tabular}{|c|}
\hline
$\mathtt{3}$ \\ \hline
$\mathtt{1}$ \\ \hline
$\mathtt{\bar{2}}$ \\ \hline
$\mathtt{\bar{3}}$ \\ \hline
$\mathtt{\bar{4}}$ \\ \hline
$\cdot \cdot \cdot $ \\ \hline
\end{tabular}
\text{ and }\mathcal{C}_{2}^{\prime }=
\begin{tabular}{|c|}
\hline
$\mathtt{6}$ \\ \hline
$\mathtt{3}$ \\ \hline
$\mathtt{2}$ \\ \hline
$\mathtt{\bar{1}}$ \\ \hline
$\mathtt{\bar{3}}$ \\ \hline
$\mathtt{\bar{4}}$ \\ \hline
$\cdot \cdot \cdot $ \\ \hline
\end{tabular}
.
\end{equation*}
Finally we derive $\Sigma _{\underline{s},2}(\underline{\lambda })=(\lambda
^{(0)},\widetilde{\lambda }^{(2)},\widetilde{\lambda }^{(1)})$ with $%
\widetilde{\lambda }^{(1)}=(3,2)$ and $\widetilde{\lambda }^{(2)}=(5,3,3,1)$.
\end{example}

\bigskip

\noindent \textbf{Remark: }Assume that $\underline{\lambda }=(\lambda
^{(0)},\lambda ^{(1)})$ is a bipartition such that $\underline{\lambda }$
belongs to $B_{e}^{\underline{s}}(\Lambda _{\underline{\frak{s}}}^{+})$
where the multicharge $\underline{s}=(s_{0},s_{1})$ verifies $s_{0}\leq s_{1}$.
Then the combinatorial procedure illustrated by the previous example which
permits to compute the crystal isomorphisms $\Sigma _{\underline{s},i},$
essentially reduces, up to a renormalization due to the change of labelling
of the Dynkin diagram in type $A_{e-1}^{(1)}$ (see (\ref{label_dynkin})), to
the algorithm depicted in Theorem 4.6 of \cite{jac}.

\subsection{A non recursive characterization of the Uglov multipartitions}

Consider a multicharge $\underline{s}$ and define the multicharge $%
\underline{\widetilde{s}}$ as in (\ref{fund}).\ Then the crystals $B_{e}^{%
\underline{s}}(\Lambda _{\underline{\frak{s}}}^{+})$ and $B_{e}^{\underline{%
\widetilde{s}}}(\Lambda _{\underline{\frak{s}}}^{+})$ are isomorphic.\ For
any multipartition $\underline{\lambda }\in B_{e}^{\underline{s}}(\Lambda _{%
\underline{\frak{s}}}^{+})$, write $I(\underline{\lambda })\in B_{e}^{%
\underline{\widetilde{s}}}(\Lambda _{\underline{\frak{s}}}^{+})$ for its
image under this crystal isomorphism.\ It is possible to obtain $I(%
\underline{\lambda })$ from $\underline{\lambda }$ by using results of \S 
\ref{action_eleme}.\ We keep the notations of \S \ref{subsec_aff}. For any $%
i=0,...,\ell -1,$ we have $z_{i}=\xi ^{(\ell -i)}\tau ^{(i)}$.\ This permits
to compute $\Gamma _{\underline{s},z_{i}(\underline{s})}(\underline{\lambda }%
)$ by using Propositions \ref{prop_isotau} and \ref{prop_iso_sigmai}.\ We
have then 
\begin{equation*}
I(\underline{\lambda })=\Gamma _{\underline{s},w(\underline{s})}(\underline{%
\lambda })
\end{equation*}
with the notation (\ref{reduc}).

\noindent Conversely, given any FLOTW multipartition $\underline{\mu }$ and
any multicharge $\underline{s}$, one can compute the multipartition $%
\underline{\lambda }\in B_{e}^{\underline{s}}(\Lambda _{\underline{\frak{s}}%
}^{+})$ such that $I(\underline{\lambda })=\underline{\mu }.\;$Indeed, we
have then $\underline{\lambda }=\Gamma _{\widetilde{\underline{s}},w^{-1}(%
\underline{\widetilde{s}})}(\underline{\mu })$.\ By remark $\mathrm{(i)}$
following Proposition \ref{prop_FLOTW}, we thus derive a non recursive
combinatorial description of the Uglov multipartitions labelling $%
B_{e}^{+}(\Lambda _{\underline{\frak{s}}}^{+})_{n}=\{\underline{\lambda }\in
B_{e}^{\underline{s}}(\Lambda _{\underline{\frak{s}}}^{+})\mid \left| 
\underline{\lambda }\right| =n\}.$

\begin{proposition}
\label{pro_non_recu}For any multicharge $\underline{s}$%
\begin{equation*}
B_{e}^{+}(\Lambda _{\underline{\frak{s}}}^{+})_{n}=\{\Gamma _{\widetilde{%
\underline{s}},w^{-1}(\underline{\widetilde{s}})}(\underline{\mu })\mid 
\underline{\mu }\in B_{e}^{\underline{\widetilde{s}}}(\Lambda _{\underline{%
\frak{s}}}^{+})_{n}\}
\end{equation*}
where $w$ is obtained from $\underline{s}$ as in (\ref{reduc}).
\end{proposition}

\subsection{Isomorphism class of a multipartition}

Suppose that $e$ is a fixed positive integer and $\underline{s}$ a
multicharge of level $\ell .$ Consider $\underline{\lambda }$ a
multipartition in $B_{e}^{\underline{s}}(\Lambda _{\underline{\frak{s}}%
}^{+}).\;$The isomorphism class of $\underline{\lambda }$ is the set 
\begin{equation*}
\frak{C}(\underline{\lambda })=\{\Gamma _{\underline{s},\underline{s}%
^{\prime }}(\underline{\lambda })\mid s^{\prime }\in \mathcal{C}(\underline{s%
})\}.
\end{equation*}
Thus $\frak{C}(\underline{\lambda })$ is the set of all multipartitions $%
\underline{\mu }$ which appear at the same place as $\underline{\lambda }$
in a crystal $B_{e}^{\underline{s}^{\prime }}(\Lambda _{\underline{\frak{s}}%
}^{+})$ where $\underline{s}^{\prime }$ is a multicharge of the orbit of $%
\underline{s}$ under the action of $\widehat{S}_{\ell }$. Then $\frak{C}(%
\underline{\lambda })$ can be determined from $\underline{\lambda }$ by
applying successive elementary transformations using Propositions \ref
{prop_isotau} and \ref{prop_iso_sigmai}. Observe that for any $\underline{%
\mu }\in \frak{C}(\underline{\lambda })$ we must have $\left| \underline{%
\lambda }\right| =\left| \underline{\mu }\right| $.\ This implies in
particular that $\frak{C}(\underline{\lambda })$ is finite. The cardinality
of $\frak{C}(\underline{\lambda })$ is in general rather complicated to
evaluate without computing the whole class $\frak{C}(\underline{\lambda })$.
Nevertheless, we are going to see in Theorem \ref{Th_bound}, that it is
possible to obtain an upper bound for $\mathrm{card}(\frak{C}(\underline{%
\lambda }))$ and to determinate a finite subset $S_{\underline{\lambda }}$
of $\widehat{S}_{\ell }$ such that 
\begin{equation*}
\frak{C}(\underline{\lambda })=\{\Gamma _{\underline{s},\underline{s}%
^{\prime }}(\underline{\lambda })\mid s^{\prime }\in S_{\underline{\lambda }%
}\cdot \underline{s}\}.
\end{equation*}

\bigskip

\begin{lemma}
\label{lemma_asym}Let $\underline{\lambda }$ be a multipartition of rank $n$%
. Assume that $\underline{s}$ is a multicharge of level $\ell $ verifying : 
\begin{equation}
s_{j}-s_{j+1}>n-1  \label{assup}
\end{equation}
for $j=0,...,\ell -2.\;$Then for any $k\in \{0,...,\ell -1\}$ we have $%
\Gamma _{\underline{s},z_{k}(\underline{s})}(\underline{\lambda })=%
\underline{\lambda }.$
\end{lemma}

\begin{proof}
Let $\underline{s}=(s_{0},...,s_{\ell -1})$ be such that $s_{j}-s_{j+1}>n-1$
for $j=0,...,\ell -2$ and let $\underline{\mu }$ be a multipartition in $%
B_{e}^{\underline{s},+}$ such that $\left| \underline{\mu }\right| \leq n.$
Let $i\in \{0,1,...,e-1\}$ and consider $\gamma _{1}=(a_{1},b_{1},c_{1})$
and $\gamma _{2}=(a_{2},b_{2},c_{2})$ two $i$-nodes in $\underline{\mu }$
such that $c_{1}<c_{2}$. We have $%
b_{1}-a_{1}+s_{c_{1}}-(b_{2}-a_{2}+s_{c_{2}})>b_{1}-a_{1}-(b_{2}-a_{2})+n-1%
\geq 0$. Hence, the contains of $\gamma _{1}$ and $\gamma _{2}$ considered
as nodes of $\underline{\mu }\in B_{e}^{\underline{s},+}$ are such that $%
c(\gamma _{1})>c(\gamma _{2})$. Hence we have $\gamma _{2}\prec _{\underline{%
s}}^{+}\gamma _{1}$.

Now, put $(s_{0}^{\prime },...,s_{l-1}^{\prime }):=z_{k}(\underline{s}%
)=(s_{0}+e,....,s_{k}+e,s_{k+1},...,s_{\ell -1}).$ As we have $s_{j}^{\prime
}-s_{j+1}^{\prime }>n-1$ for $j=0,...,\ell -2$, the above discussion shows
that the order $\prec _{\underline{s}}^{+}$ and $\prec _{z_{k}(\underline{s}%
)}^{+}$ on the $i$-nodes of $\underline{\mu }$ coincide.
\end{proof}

\bigskip

\begin{theorem}
\label{Th_bound}Suppose that $e$ is a fixed positive integer and $\underline{%
s}$ a multicharge of level $\ell $ such that 
\begin{equation*}
0\leq s_{\ell -1}\leq \cdot \cdot \cdot \leq s_{0}<e
\end{equation*}
Consider $\underline{\lambda }$ a FLOTW multipartition in $B_{e}^{\underline{%
s}}(\Lambda _{\underline{\frak{s}}}^{+})$ of order $n.\;$For any $j\in
\{0,...,\ell -2\}$, let $p_{j}$ be the minimal nonnegative integer such that 
\begin{equation*}
s_{j}+p_{j}e-s_{j+1}>n-1.
\end{equation*}
Then we have : 
\begin{equation*}
\frak{C}(\underline{\lambda })=\{\Gamma _{\underline{s},z_{0}^{a_{0}}...z_{%
\ell -2}^{a_{\ell -2}}\sigma (\underline{s})}(\underline{\lambda })\ |\sigma
\in S_{\ell }\text{ and }0\leq a_{j}\leq p_{j}\text{ for any }j\in
\{0,...,\ell -2\}\}.
\end{equation*}
In particular, $\frak{C}(\underline{\lambda })$ is finite and 
\begin{equation*}
\mathrm{card}(\frak{C}(\underline{\lambda }))\leq \ell !\prod_{j=0}^{\ell
-2}(p_{j}+1).
\end{equation*}
\end{theorem}

\begin{proof}
Consider $\underline{\mu }\in \frak{C}(\underline{\lambda }).\;$According to
(\ref{reduc}), one can write 
\begin{equation*}
\underline{\mu }=\Gamma _{\underline{s},w^{-1}(\underline{s})}(\underline{%
\lambda })\text{ with }w^{-1}=z_{\ell -1}^{-k}\sigma z_{\ell -2}^{r_{\ell
-2}}\cdot \cdot \cdot z_{0}^{r_{0}}.
\end{equation*}
By Remark $\mathrm{(ii)}$ following Example \ref{exa_restr}, we have $%
\underline{\mu }=\Gamma _{\underline{s},w^{-1}(\underline{s})}(\underline{%
\lambda })=\Gamma _{\underline{s},u(\underline{s})}(\underline{\lambda })$
where $u=\sigma z_{\ell -2}^{r_{\ell -2}}\cdot \cdot \cdot z_{0}^{r_{0}}.$
By Lemma \ref{lemma_asym}, we can thus derive 
\begin{equation*}
\frak{C}(\underline{\lambda })=\{\Gamma _{\underline{s},z_{0}^{a_{0}}...z_{%
\ell -2}^{a_{\ell -2}}\sigma (\underline{s})}(\underline{\lambda })\ |\sigma
\in S_{\ell }\text{ and }0\leq a_{j}\leq p_{j}\text{ for any }j\in
\{0,...,\ell -2\}\}
\end{equation*}
and our theorem follows.
\end{proof}

\bigskip

\noindent \textbf{Acknowledgments}: The authors thank the organizers of the
workshop ``Autour des conjectures de Brou\'{e}'' hold at the CIRM in Luminy
(from 05/27/07 to 06/02/07) during which this paper was completed.

\end{document}